\documentclass[12pt,naturalnames]{article}
\date{}
\title{\vspace*{-36pt}
Bipolar orientations on planar maps and SLE$_{12}$ }
\author{
\hspace*{-200pt}\begin{tabular}{c}Richard Kenyon\\[-5pt]\small Brown\end{tabular}\hfill
\begin{tabular}{c}Jason Miller\\[-5pt]\small Cambridge\end{tabular}\hfill
\begin{tabular}{c}Scott Sheffield\\[-5pt]\small MIT\end{tabular}\hfill
\begin{tabular}{c}David B.\! Wilson\\[-5pt]\small Microsoft\end{tabular}\hspace*{-200pt}}
\usepackage[textheight=610pt]{geometry}
\usepackage{mathtools}
\usepackage[usenames,dvipsnames]{xcolor}
\usepackage[pdftitle={Bipolar orientations on planar maps and SLE(12)},pdfauthor={Richard Kenyon, Jason Miller, Scott Sheffield, and David B. Wilson},colorlinks=true,linkcolor=NavyBlue,urlcolor=RoyalBlue,citecolor=PineGreen,bookmarks=true,bookmarksopen=true,bookmarksopenlevel=2,unicode=true]{hyperref}
\usepackage[font=footnotesize]{caption}
\usepackage{amssymb}
\usepackage{amsthm}
\usepackage{amsfonts}
\usepackage{tikz}
\tikzset{>=latex}
\usepackage{subcaption}
\usepackage{graphicx}
\usepackage{tabularx}
\usepackage{microtype}
\usepackage{enumerate}
\usepackage{xfrac}
\newcommand{\old}[1]{}
\newcommand{\Var}{\operatorname{Var}}
\newcommand{\R}{{\mathbb R}}
\newcommand{\Z}{{\mathbb Z}}
\newcommand{\E}{{\mathbb E}}
\newcommand{\eps}{\varepsilon}
\newcommand{\be}{\begin{equation}}
\newcommand{\ee}{\end{equation}}

\newtheorem{theorem}{Theorem}
\newtheorem{conjecture}{Conjecture}

\newtheorem{prop}[theorem]{Proposition}
\newtheorem{corollary}[theorem]{Corollary}

\newtheorem{remark}{Remark}

\newcommand{\s}{{\mathbb S}}
\renewcommand{\C}{{\mathbb C}}

\newcommand{\wh}{\widehat}

\def\@rst #1 #2other{#1}
\newcommand\MR[1]{\relax\ifhmode\unskip\spacefactor3000 \space\fi
  \MRhref{\expandafter\@rst #1 other}{#1}}
\newcommand{\MRhref}[2]{\href{http://www.ams.org/mathscinet-getitem?mr=#1}{MR#1}}

\newcommand{\SLE}{\operatorname{SLE}}
\newcommand{\CLE}{\operatorname{CLE}}

\begin{document}

\maketitle
\begin{abstract}
  We give bijections between bipolar-oriented (acyclic with unique
  source and sink) planar maps and certain random walks, which show that
  the uniformly random bipolar-oriented
  planar map, decorated by the ``peano curve'' surrounding the tree of
  left-most paths to the sink, converges in law with respect to the
  peanosphere topology to a $\sqrt{4/3}$-Liouville quantum gravity
  surface decorated by an independent Schramm-Loewner evolution with
  parameter~$\kappa=12$ (i.e., $\SLE_{12}$).  This result is universal in the sense
  that it holds for bipolar-oriented triangulations, quadrangulations,
  $k$-angulations, and maps in which face sizes are mixed.
\end{abstract}

\section{Introduction}
\subsection{Planar maps}

A \textit{planar map\/} is a planar graph together with an embedding into $\R^2$ so that no two edges cross. More precisely, a planar map is an equivalence class of such embedded graphs, where two embedded graphs are said to be equivalent if there exists an orientation preserving homeomorphism $\R^2 \to \R^2$ which takes the first to the second.  The enumeration of planar maps started in the 1960's in work of Tutte \cite{tutte:planar-maps}, Mullin \cite{mullin:tree-maps}, and others.  In recent years, new combinatorial techniques for the analysis of random planar maps, notably via random matrices and tree bijections, have revitalized the field.  Some of these techniques were motivated
from physics, in particular from
conformal field theory and string theory.

There has been significant mathematical progress on the enumeration and scaling limits of random planar maps chosen uniformly from the set of all rooted planar maps with a given number of edges, beginning with the bijections of Cori--Vauquelin \cite{cori-vauquelin} and Schaeffer \cite{schaeffer} and progressing to the existence of Gromov--Hausdorff metric space limits established by Le Gall \cite{legall:bm} and Miermont \cite{miermont:bm}.

There has also emerged a large literature on
planar maps that come equipped with additional structure, such as the instance of a model from statistical physics, e.g., a uniform spanning tree, or an Ising model configuration.
  These ``decorated planar maps'' are important in
  Euclidean 2D statistical physics.  The reason is that it is often
  easier to compute ``critical exponents'' on planar maps than on
  deterministic lattices. Given the planar map exponents, one can
  apply the KPZ formula to \textit{predict\/} the analogous Euclidean
  exponents.\footnote{This idea was used by Duplantier to derive the so-called
  Brownian intersection exponents \cite{duplantier:brownian}, whose
  values were subsequently verified mathematically by Lawler, Schramm,
  and Werner \cite{LSW:brownian1,LSW:brownian2,LSW:brownian3} in an
  early triumph of Schramm's SLE theory \cite{schramm:sle}.  An
  overview with a long list of references can be found in
  \cite{duplantier-sheffield:KPZ}.}
  In this paper, we
consider random planar maps equipped with bipolar orientations.

\subsection{Bipolar and harmonic orientations}  \label{subsec::bporpm}
A \textit{bipolar (acyclic) orientation\/} of a graph~$G$ with specified source and sink (the ``poles'')
is an acyclic orientation of its edges with no source or sink except at the specified poles.
(A \textit{source\/} (resp.\ \textit{sink\/}) is a vertex with no incoming (resp.\ outgoing) edges.)
For any graph~$G$ with adjacent source and sink,
bipolar orientations are counted by
the coefficient of $x$ in the Tutte polynomial $T_G(x,y)$,
which also equals the coefficient
of $y$ in $T_G(x,y)$; see \cite{bipolarorientationsrevisited} or the overview in \cite{FPS}.
In particular, the number of bipolar orientations does not depend on the choice of source and sink as long as they are adjacent.
When the source and sink are adjacent, there are bipolar orientations
precisely when the graph is biconnected, i.e., remains connected after the removal of any vertex \cite{lempel-even-cederbaum}.
If the source and sink are not adjacent,
adjoining an edge between the source and sink does not affect the number of bipolar orientations,
so bipolar orientations are counted by these Tutte
coefficients in the augmented graph.

Let $G$ be a finite connected planar map, with no self-loops but with multiple edges
allowed, with a specified source and sink that are incident to the
same face.  It is convenient to embed $G$ in the disk so that the
source is at the bottom of the disk (and is denoted S, for south
pole), the sink is at the top (and is denoted N, for north pole), and all
other vertices are in the interior of the disk (see Figure~\ref{fig::ao3}).
Within the disk there are two faces that are boundary faces,
which can be called W (the west pole) and E (the east pole).
Endowing $G$ with a bipolar orientation is a way to endow it and its
dual map $G^*$ with a coherent notion of ``north, south, east, and west'':
one may define the directed edges to point \textit{north}, while their
opposites point \textit{south}.
Each primal edge has a face to its \textit{west\/} (left when facing north) and its \textit{east\/} (right), and dual edges are oriented in the westward direction (Figure~\ref{fig::ao3}).

\begin{figure}[t!!]
\begin{center}
\includegraphics[scale=1]{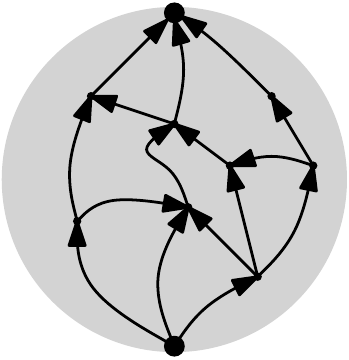}\hfill\includegraphics[scale=1]{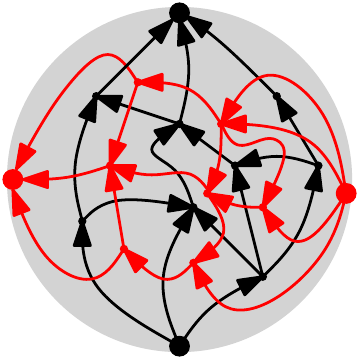}\hfill\includegraphics[scale=1]{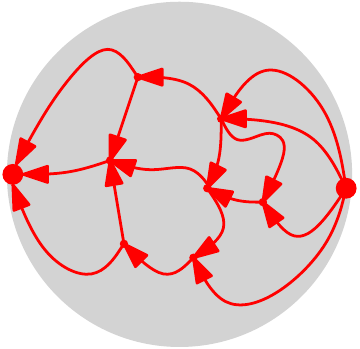}
\end{center}
\caption{\label{fig::ao3} {\bf Left:} A planar map embedded in a disk with two boundary vertices, with a north-going bipolar orientation. {\bf Right:} The dual bipolar-oriented planar map, which has two boundary dual vertices on the disk. {\bf Middle:} Primal and dual bipolar-oriented maps together.  The dual orientations are obtained from the primal orientations by rotating the arrows left.}
\end{figure}

\enlargethispage{12pt}
Given an orientation of a finite connected planar map $G$, its dual
orientation of $G^*$ is obtained by rotating directed edges
counterclockwise.  If an orientation has a sink or source at an
interior vertex, its dual has a cycle around that vertex.  Suppose an
orientation has a cycle but has no source or sink at interior
vertices.  If this cycle surrounds more than one face, then one can
find another cycle that surrounds fewer faces, so there is a cycle
surrounding just one face, and the dual orientation has either a
source or sink at that (interior) face.  Thus an orientation of $G$
is bipolar acyclic precisely when its dual orientation of $G^*$ is
bipolar acyclic.  The east and west poles of $G^*$ are the source
and sink respectively of the dual orientation (see Figure~\ref{fig::ao3}).

One way to construct bipolar orientations is via electrical networks.  Suppose every edge of $G$ represents a conductor with some generic positive conductance, the south pole is at 0 volts, and the north pole is at 1 volt.  The voltages are harmonic except at the boundary vertices, and for generic conductances, provided every vertex is part of a simple path connecting the two poles, the interior voltages are all distinct.  The \textit{harmonic orientation\/} orients each edge towards its higher-voltage endpoint.  The harmonic orientation is clearly acyclic, and by harmonicity, there are no sources or sinks at interior vertices.
In fact, for any planar graph with source and sink incident to the same face, \textit{any\/} bipolar orientation is the harmonic orientation for some suitable choice of conductances on the edges \cite[Thm.~1]{abrams-kenyon:fixed-energy}, so for this class of graphs, bipolar orientations are equivalent to harmonic orientations.

Suppose that a bipolar-oriented planar map~$G$ has an interior vertex
incident to at least four edges, which in cyclic order are oriented
outwards, inwards, outwards, inwards.  By the source-free sink-free
acyclic property, these edges could be extended to oriented paths
which reach the boundary, and by planarity and the acyclic property,
the paths would terminate at four distinct boundary vertices.  Since
(in this paper) we are assuming that there are only two boundary
vertices, no such interior vertices exist.  Thus at any interior
vertex, its incident edges in cyclic order consist of a single group
of north-going edges followed by a single group of south-going edges,
and dually, at each interior face the edges in cyclic order consist of
a single group of clockwise edges followed by a single group of
counterclockwise edges.

In particular, each vertex (other than the north pole) has a unique
``west-most north-going edge,''
which is its NW edge.  The \textit{NW tree\/} is the directed
tree which
maps each vertex (other than the north pole) to its NW edge, and maps
each edge to the vertex to its north.
Geometrically, the NW tree can be drawn so that each NW edge is entirely
in the NW tree, and for
each other edge, a segment containing the
north endpoint of the edge is in the NW tree (see Figure~\ref{fig::ao}).
We define southwest, southeast, and northeast trees similarly.

We will exhibit (see Theorems~\ref{marked-map} and \ref{general}) a
bijection between bipolar-oriented planar maps (with given face-degree
distribution) and certain types of random walks in the nonnegative
quadrant $\Z_{\geq 0}^2$.  This bijection leads to exact enumerative formulae
as well as local information about the maps such as degree
distributions.  For previous enumerative work on this model, including
bijections between bipolar-oriented planar maps and other objects, see
e.g. \cite{FPS,bousquet-melou:maps,BBMF:baxter-bipolar,FFNO:baxter}.

\subsection{SLE and LQG}
After the bijections our second main result is the identification of the scaling limit of the bipolar-oriented map
with a \textit{Liouville quantum gravity\/} (LQG) surface decorated by a \textit{Schramm-Loewner evolution\/} (SLE) curve,
see Theorem~\ref{scalinglimit}.

We will make use of the fact proved in \cite{DMS:mating,  miller-sheffield:finite-trees, GHMS:covariance} that an SLE-decorated LQG surface can be equivalently defined as a mating of a correlated pair of continuum random trees (a so-called \textit{peanosphere}; see Section~\ref{peanosphere}) where the correlation magnitude is determined by parameters that appear in the definition of LQG and SLE (namely $\gamma$ and $\kappa'$).

The scaling limit result can thus be formulated as the statement that a certain pair of discrete random trees determined by the bipolar orientation (namely the \textit{northwest\/} and \textit{southeast\/} trees, see Section~\ref{subsec::bporpm}) has, as a scaling limit, a certain correlated pair of continuum random trees.
Although LQG and SLE play a major role in our motivation and intuition (see Sections~\ref{peanosphere} and~\ref{imaginarysection}), we stress that no prior knowledge about these objects is necessary to understand either the main scaling limit result in the current paper or the combinatorial bijections behind its proof (Sections~\ref{sec::bijections} and~\ref{sec::triangle}).

Before we move on to the combinatorics, let us highlight another point about the SLE connection.
There are several special values of the parameters $\kappa$ and $\kappa'=16/\kappa$ that are related to discrete statistical physics models.
($\SLE_\kappa$ with $0<\kappa\le 4$ and $\SLE_{16/\kappa}$ are closely related \cite{zhan:duality,dubedat:duality,miller-sheffield:ig1,miller-sheffield:ig4}, which is known
as SLE-duality.)  These special $\{\kappa,\kappa'\}$ pairs include $\{2,8\}$ (for loop-erased random walk and the uniform spanning tree) \cite{LSW:tree}, $\{8/3,6\}$ (for percolation and Brownian motion) \cite{smirnov:percolation,LSW:brownian-frontier}, $\{3,16/3\}$ (for the Ising and FK-Ising model) \cite{smirnov:ising,CDCHKS:ising}, and
$\{4,4\}$ (for the Gaussian free field contours)
\cite{schramm-sheffield:discrete-GFF,schramm-sheffield:continuum-GFF}.
The relationships between these special $\{\kappa, \kappa'\}$ values and the corresponding discrete models were all discovered
or
 conjectured within a couple of years of Schramm's introduction of SLE,
building on earlier arguments from the physics literature.  We note that all of these relationships have random planar map analogs, and
that they all correspond to $\{\kappa, \kappa' \} \subset [2,8]$. This range is significant because the so-called \textit{conformal loop ensembles\/} $\CLE_\kappa$ \cite{sheffield:cle-trees,sheffield-werner:markovian} are only defined for $\kappa \in (8/3,8]$, and the discrete models mentioned above are all related to random collections of loops in some way, and hence have either $\kappa$ or $\kappa'$ in the range $(8/3,8]$.  Furthermore, it has long been known that ordinary $\SLE_{\kappa}$ does not have time reversal symmetry when $\kappa > 8$ \cite{rs:k_8_reverse} (see \cite{miller-sheffield:ig4} for the law of the time-reversal of such an $\SLE_\kappa$ process),
and it was thus widely assumed that discrete statistical physics systems
would not converge to $\SLE_\kappa$ for $\kappa>8$ \cite{Cardy}.

In this paper the relevant $\{\kappa, \kappa' \}$ pair is $\{4/3,12\}$. This special pair is interesting
in part because it lies outside the range $[2,8]$.
It has been proposed, based on heuristic arguments and simulations, that ``activity-weighted'' spanning trees
should have SLE scaling limits with $\kappa$ anywhere in the range $[4/3,4)$ and $\kappa'$ anywhere in the range $(4,12]$
\cite{kassel-wilson:active}.
In more recent work, subsequent to our work on bipolar orientations,
using a generalization of the inventory accumulation model in \cite{sheffield:burgers},
the activity-weighted spanning trees on planar maps were shown to converge to SLE-decorated LQG in the peanosphere topology for
this range of $\kappa,\kappa'$ \cite{GKMW:active-tree-map}.

We will further observe that if one modifies the bipolar orientation model by a weighting that makes the faces more or less balanced (in terms of their number of clockwise and counterclockwise oriented boundary edges), one can obtain any $\kappa \in (0,2)$ and any $\kappa' \in (8,\infty)$.
In a companion to the current paper \cite{KMSW2}, we discuss
a different generalization of bipolar orientations that we conjecture
gives SLE for $\kappa\in[12-8\sqrt{2},4)$ and $\kappa'\in(4,12+8\sqrt{2}]$.

In this article we consider an opposite pair of trees (NW-tree and SE-tree).
It is also possible to consider convergence of all four trees (NW, SE, NE, and SW) simultaneously:
this is done in the recent article \cite{GNS:bipolar}.

\subsection{Outline}

In Sections~\ref{sec::bijections} and~\ref{sec::triangle} we establish our combinatorial results and describe the scaling limits of the NW and SE trees in terms of a certain two-dimensional Brownian excursion. In Section~\ref{sec::peanosphere}
we explain how this implies that the uniformly random bipolar-oriented map with $n$ edges,
and fixed face-degree distribution, decorated by its NW tree, converges in law as $n \to \infty$ to a $\sqrt{4/3}$-Liouville quantum gravity sphere decorated by space-filling $\SLE_{12}$ from $\infty$ to $\infty$.
This means that, following the curve which winds around the NW tree,  the distances to the N and S poles scale to an appropriately correlated pair of Brownian excursions.
We also prove a corresponding \textit{universality\/} result: the above scaling limit holds for essentially any distribution on face degrees (or, dually, vertex degrees) of the random map.

In Section~\ref{imaginarysection} we explain, using the so-called imaginary geometry theory, what is special about the value $\kappa'=12$. These observations allow us to explain at a heuristic level why (even before doing any discrete calculations) one would \textit{expect\/} $\kappa'=12$ to arise as the scaling limit of bipolar orientations.

\medskip
\enlargethispage{18pt}

\noindent{\bf{Acknowledgements.}}
R.K.\ was supported by NSF grant DMS-1208191 and Simons Foundation grant 327929.
J.M.\ was supported by NSF grant DMS-1204894.
S.S.\ was supported by a Simons Foundation grant,
NSF grant DMS-1209044, and EPSRC grants {EP/L018896/1} and {EP/I03372X/1}.
We thank the Isaac Newton Institute for Mathematical Sciences
for its support and hospitality during the program on Random Geometry,
where this work was initiated. We thank Nina Holden for comments on a draft of this paper.

\section{Bipolar-oriented maps and lattice paths} \label{sec::bijections}

\subsection{From bipolar maps to lattice paths}

For the bipolar-oriented planar map in Figure~\ref{fig::ao3},
Figure~\ref{fig::ao} illustrates its \textit{NW tree\/} (in red), \textit{SE tree\/} (in blue), and the \textit{interface path\/} (in green) which winds between them from the south pole to the north pole.
The interface path
has two types of steps:
\begin{enumerate}
\item Steps that traverse an edge (between red and blue sides).
\item Steps that traverse an interior face from its maximum to its minimum.
Face steps can be subcategorized according to the number of edges on the west and east sides of the face, where the maximum and minimum vertex of a face separate its west from its east.  If a face has $i+1$ edges on its west and $j+1$ edges on its east, we say that it is of type $(i,j)$.
\end{enumerate}
Observe that each face step has edge steps immediately before and after it.

\begin{figure}[t!!]
\begin{center}
\includegraphics[scale=1.6]{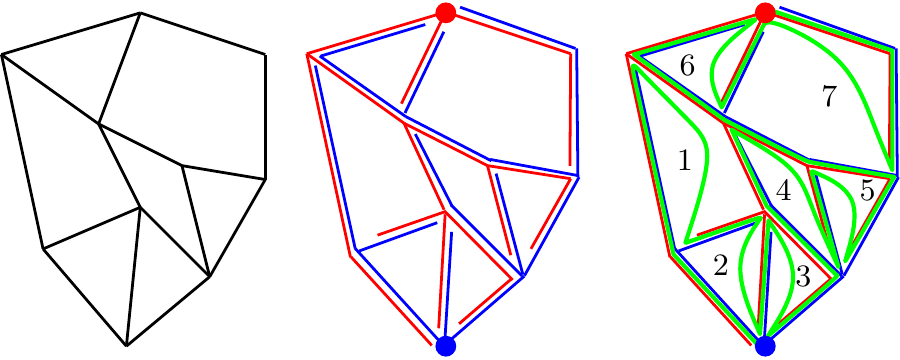}
\caption {\label{fig::ao} {\bf Left:} A map with a bipolar orientation, embedded so each edge is oriented ``upward'' (i.e., in the direction along which the vertical coordinate increases).
{\bf Middle:} Set of oriented edges can be understood as a tree, the \textit{northwest tree},
where the parent of each edge is the leftmost upward oriented edge it can merge into.  If we reverse the orientations of all edges, we can define an analogous tree (blue) and embed both trees (using the British convention of driving on the left side) so that they don't cross each other. {\bf Right:} We then add a green path tracing the interface between the two trees. Each edge of the interface
moves along an edge of the map
or across a face of the map.
For illustration purposes, faces are numbered by the order they are traversed by the green path,
but it is the traversals of the edges of the green path that correspond to steps of the lattice path.
}
\end{center}
\end{figure}

Let $E$ be the set of edges of the planar map, which we order $e_0,\ldots,e_{|E|-1}$ according to the green path going from
the south pole S to the north pole N.  For each edge $e_t$, let $X_t$ be distance in the blue tree from the blue root (S) to the
lower endpoint of $e_t$, and let $Y_t$ be the distance in the red tree from the red root (N) to the upper endpoint of $e_t$.
Suppose the west outer face has $m+1$ edges and the east outer face has $n+1$ edges.  Then the sequence $\{(X_t,Y_t)\}_{0\leq t\leq|E|-1}$ defines a walk or lattice path that starts at $(0,m)$ when $t=0$ and ends at $(n,0)$ when $t=|E|-1$, and which remains in the nonnegative quadrant.  If there is no face step between $e_t$ and $e_{t+1}$, then the walk's increment $(X_{t+1},Y_{t+1})-(X_t,Y_t)$ is $(1,-1)$.  Otherwise there is exactly one face step between $e_t$ and $e_{t+1}$; if that face has $i+1$ edges on its west and $j+1$ edges on its east, then the walk's increment is $(-i,j)$, see Figure~\ref{fig::stepdiagram}.

\begin{figure}[ht!!]
\begin{center}
\includegraphics[scale=1]{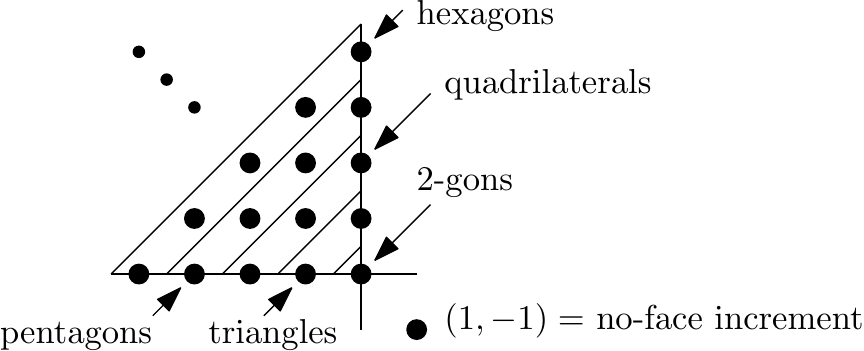}
\caption {\label{fig::stepdiagram} Lattice path increments.}
\end{center}
\end{figure}

For the example in Figure~\ref{fig::ao}, the walk starts at $(0,2)$ and ends at $(3,0)$.

\subsection{From lattice paths to bipolar maps}
The above construction can be reversed, constructing a
bipolar-oriented planar map from a lattice path of the above type.

We construct the bipolar-oriented planar map by sewing edges and
oriented polygons according to the sequence of steps of the path.  Let
$m_{i,j}$ denote a step of $(-i,j)$ with $i,j\geq 0$, and $m_e$ denote
a step of $(1,-1)$.

It is convenient to extend the bijection, so that
it can be applied to any sequence of these steps, not just those corresponding to walks remaining in the quadrant.  These steps give
sewing instructions to augment the current ``marked bipolar map'', which will be a slightly more general object.

A marked bipolar map is a bipolar-oriented planar map together with a ``start
vertex'' on its western boundary which is not at the top, and an
``active vertex'' on its eastern boundary which is not at the bottom,
such that the start vertex and every vertex below it on the western
boundary has at most one downward edge, and the active vertex and
every vertex above it on the eastern boundary has at most one upward
edge.  We think of the edges on the western boundary below the start
vertex and on the eastern boundary above the active vertex as being
``missing'' from the marked bipolar map: they are boundaries of open faces that
are part of the map, but are not themselves in the map.

\begin{figure}[htbp]
  \includegraphics[width=\textwidth]{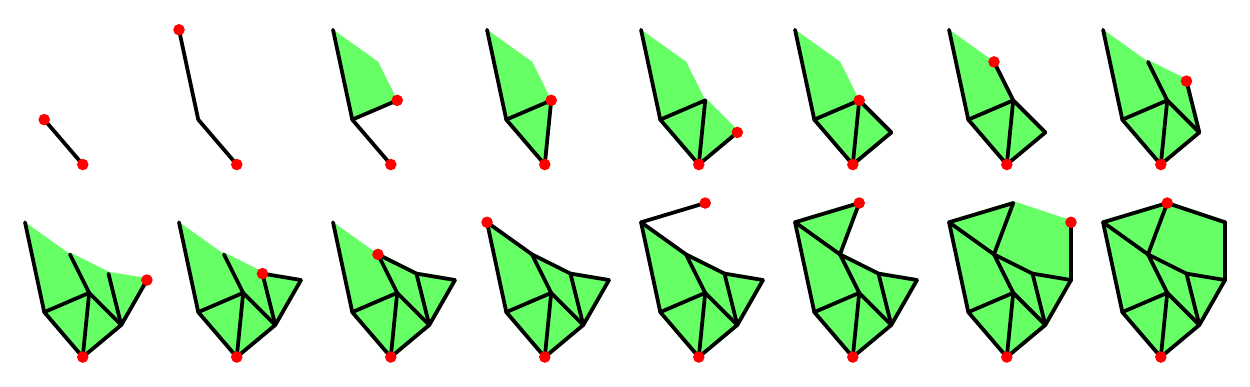}
  \caption{\label{sew} The process of sewing oriented polygons and
    edges to obtain a bipolar-oriented planar map.  The intermediate
    structures are marked bipolar-oriented planar maps, which may have
    some edges missing on the boundaries. The sequence of steps is:
    $m_e, m_{0,2},m_{1,0},m_{0,1},m_e,m_e,m_{1,1},m_{0,1},m_e,m_e,m_e,m_e,m_{1,0},m_{2,1},m_e$.}
\end{figure}

Initially the marked bipolar map consists of an oriented edge whose lower
endpoint is the start vertex and whose upper endpoint is the active
vertex.  Each $m_e$ and $m_{i,j}$ move adds exactly one edge to the
marked bipolar map.  The $m_{i,j}$ moves also add an open face.

The $m_e$ moves will sew an edge to the current marked bipolar map
upwards from the active vertex and move the active vertex to the upper
endpoint of the new edge.  If the eastern boundary had a vertex above
the active vertex, the new edge gets sewn to the southernmost missing
edge on the eastern boundary, and otherwise there is a new vertex
which becomes the current top vertex.

The $m_{i,j}$ moves will sew an open
face with $i+1$ edges on its west and $j+1$ edges on its east, sewing
the north of the face to the active vertex and the west of the face to
the eastern boundary of the marked bipolar map, and then sew an edge to the
southernmost east edge of the new face; the new active vertex is the upper vertex of this edge.
We can think of $m_{i,j}$ as being composed of two submoves,
a move $f_{i,j}$ which sews the open polygon to the structure,
with the top of the polygon at the old active vertex, and with
the new active vertex at the bottom of the polygon, followed by
a regular $m_e$ move.
If there are fewer than $i+1$
edges below the (old) active vertex, then the new face gets sewn to as many
of them as there are, the start vertex is no longer at the bottom, and
the remaining western edges of the face are missing from the map. As seen in the proof of Theorem \ref{general} below,
this happens when the walk goes out of the positive quadrant; these western edges will remain missing for any subsequent steps.

The final marked bipolar map is considered a (unmarked) bipolar-oriented
planar map if the start vertex is at the south and the active vertex
is at the north, or equivalently, if there are no missing edges.

\begin{theorem}\label{marked-map}
  The above mapping from sequences of moves from $\{m_e\}\cup\{m_{i,j}:i,j\geq0\}$
  to marked bipolar maps is a bijection.
\end{theorem}
\begin{proof}
  Consider a marked bipolar map obtained from a sequence of moves.
  The number of edges present in the structure determines the length
  of the sequence.  If that length is positive, then the easternmost
  downward edge from the active vertex was the last edge adjoined to
  the structure.  If this edge is the southernmost edge on the eastern
  boundary of a face, then the last move was one of the $m_{i,j}$'s, and
  otherwise it was $m_e$.  Since the last move and preceding structure
  can be recovered, the mapping is an injection.

  Starting from an arbitrary marked bipolar map, we can inductively
  define a sequence of moves as above by considering the easternmost
  downward edge from the active vertex.  This sequence of moves yields
  the original marked bipolar map, so the mapping is a surjection.
\end{proof}

Next we restrict this bijection to sequences of moves which give valid
bipolar-oriented planar maps.  A sequence of moves can of course be
encoded as a path.

\begin{theorem}\label{general}
  The above mapping gives a bijection from length-$(\ell-1)$ paths
  from $(0,m)$ to $(n,0)$ in the nonnegative quadrant having
  increments $(1,-1)$ and $(-i,j)$ with $i,j\geq 0$, to
  bipolar-oriented planar maps with $\ell$ total edges and $m+1$ and
  $n+1$ edges on the west and east boundaries respectively.  A step of
  $(-i,j)$ in the walk corresponds to a face with degree $i+j+2$ in
  the planar map.
\end{theorem}

\begin{proof}
  When we make a walk $(X_t,Y_t)_{t\geq 0}$ in $\Z^2$ started from $(X_0,Y_0)$
   using these moves, not necessarily confined to the quadrant, by induction
\[\begin{aligned}
X_t-X_0 = -1\,+\,
&\text{(\# non-missing edges on the eastern boundary)}\\
-\,&\text{(\# missing edges on the western boundary)}
\end{aligned}\]
and
\[\begin{aligned}
Y_t-Y_0 = 1\,+\,
&\text{(\# missing edges on the eastern boundary)}\\
-\,&\text{(\# non-missing edges on the western boundary)}\,.
\end{aligned}\]

  When the walk is started at $(0,m)$, the
  start vertex remains at the south pole precisely when the first coordinate
  always remains nonnegative.  In this case, there are no missing edges on the western boundary, so the final number of non-missing edges
  on the eastern boundary is $n+1$.

  Suppose that we reverse the sequence of moves, and replace each $m_{i,j}$
  with $m_{j,i}$, to obtain a new sequence.  Write each $m_{i,j}$
  as the face move $f_{i,j}$ followed by $m_e$.  Recall that the initial structure
  was an edge; we may instead view the intial structure as a vertex
  followed by an $m_e$ move.
  Written in this way, if the old sequence is $m_e^{k_0+1} f_{i_1,j_1} m_e^{k_1+1} f_{i_2,j_2}\cdots m_e^{k_q+1}$,
  the new sequence is $m_e^{k_q+1} \cdots f_{j_2,i_2} m_e^{k_1+1} f_{j_1,i_1} m_e^{k_0+1}$.
  We then see that the structure obtained from the new
  sequence is the same as the structure obtained from the old
  sequence but rotated by $180^\circ$, and with the roles of start
  and active vertices reversed.

  Using this reversal symmetry with our previous observation,
  it follows that the active vertex is at the north pole
  precisely when the second coordinate achieves its minimum
  on the last step (it may also achieve its minimum earlier),
  and the number of (non-missing) edges on the western boundary is $m+1$.
\end{proof}

If we wish to restrict the face degrees, the bijection continues to
hold simply by restricting the set of allowed steps of the paths.

We can use the bijection to prove the following result:
\begin{theorem}
  Any finite bipolar-oriented planar map which has no self-loops or
  pairs of vertices connected by multiple edges has a straight-line
  planar embedding such that edges are oriented upwards, i.e., in the
  direction of increasing $y$-coordinate, as in Figure~\ref{fig::ao}.
\end{theorem}
\begin{proof}
If the bipolar-oriented planar map has a face with more than 3 sides,
then let $v_1,\ldots,v_4$ denote four of its vertices in cyclic order.
The map could contain the edges $(v_1,v_3)$ or $(v_2,v_4)$, embedded outside
the face, but it cannot contain both of them without violating planarity.
We may adjoin an edge which the graph does not already contain, embed
it within the face, and then orient it so that the augmented planar map
is bipolar-oriented.  By repeating this process, we see that we may assume
that the map is a triangulation.

Given a bipolar-oriented triangulation without multiple edges between
vertices, we can convert it to a walk using the bijection, and then convert it back to a bipolar triangulation again using the bijection.
When converting the walk back to a triangulation,
we do so while maintaining the following geometric property:
We require that every edge, missing or non-missing, be embedded as
a straight line oriented upwards.
We also require that every pair of vertices on the right boundary of the
closure of the structure have a ``line-of-sight'' to each other,
unless the structure
contains an edge (missing or non-missing) connecting them.  By ``having a line-of-sight'', we mean that the open line segment connecting the vertices is disjoint from the closure of the structure.

It's trivial to make the initial structure satisfy the geometric property.
Edge moves trivially maintain the geometric property.
Since the graph does not contain multiple edges connecting vertices,
the move $m_{1,0}$ (adjoining a leftward triangle) connects two vertices
that are within line-of-site, so it also maintains the geometric property.
The move $m_{0,1}$ adjoins a rightward triangle and necessarily makes the right boundary non-concave.  However, for any pair of vertices on the right boundary that are within line-of-sight of each other, we may place the new vertex of the triangle sufficiently close to its left edge that the line-of-sight is not obstructed, and since there are only finitely many pairs of vertices on the right boundary, we may embed the new triangle so that the geometric property is maintained.

By induction the final structure satisfies the geometric property, so it is
a straight-line embedding with edges oriented upwards.
\end{proof}

\subsection{Path scaling limit}\label{pathscalinglimit}

What happens if we consider a random bipolar-oriented planar map such as the one in Figure~\ref{fig::ao}, where we fix the left boundary length ($3$ in Figure~\ref{fig::ao}), the right boundary length ($4$ in Figure~\ref{fig::ao}), and the total number $\ell$ of edges ($16$ in Figure~\ref{fig::ao})?
We consider the limiting case where the boundary lengths are fixed and $\ell \to \infty$.  What can one say about the limiting joint law of the pair of trees in Figure~\ref{fig::ao} in this situation?

In light of Theorem~\ref{general}, understanding this limiting law amounts to understanding the limiting law of its associated lattice path. For example, if the map is required to be a triangulation, then the lattice path is required to have increments of size $(1,-1)$, $(-1,0)$, and $(0,1)$.  Since $\ell\to\infty$ with fixed endpoints, there are $\ell/3+O(1)$ steps of each type.  One can thus consider a random walk of length $\ell-1$ with these increment sizes (each chosen with probability $1/3$) conditioned to start and end at certain fixed values, and to stay in the nonnegative quadrant.

It is reasonable to expect that if a random walk on $\Z^2$ converges to Brownian motion with some non-degenerate diffusion matrix, then the same random walk conditioned to stay in a quadrant (starting and ending at
fixed locations
when the number of steps gets large)
should scale to a form of the Brownian excursion, i.e., a Brownian bridge constrained to stay in the same quadrant (starting and ending at $0$).
The recent work \cite[Theorem 4]{DurajWachtel} contains a precise theorem of this form,
and Proposition~\ref{prop::constrainedbrownianbridge} below is a special case of this theorem.  (The original theorem is for walks with a diagonal covariance matrix, but supported on a generic lattice, which implies Proposition~\ref{prop::constrainedbrownianbridge} after applying a linear transformation to the lattice.)

Recall that the period of a random walk on $\Z^2$ is the smallest integer $p\geq1$
such that the random walk has a positive probability to return to zero
after $kp$ steps for all sufficiently large integers $k > 0$.

\newcommand{\pstart}{z_{\text{start}}}
\newcommand{\pend}{z_{\text{end}}}
\begin{prop} \label{prop::constrainedbrownianbridge}
Let $\nu$ be a probability measure supported on $\Z^2$ with expectation zero and moments of all orders.
Let $p\geq1$ denote the period of the random walk on $\Z^2$ with step distribution $\nu$.
Suppose that for given $\pstart,\pend\in\Z_{\geq 0}^2$, for some $\ell$ there is a positive probability path in $\Z_{\geq0}^2$ from $\pstart$ to $\pend$ with $\ell$ steps from~$\nu$.
Suppose further that for any $R>0$ there is a point $z\in\Z_{\geq 0}^2$ that is distance at least $R$ from the boundary of the quadrant, such that there is a path from $\pstart$ to $z$ to $\pend$ with steps from~$\nu$ that remains in the quadrant $\Z_{\geq 0}^2$.
For sufficiently large $n$ with $n\equiv \ell\bmod p$, consider a random walk
$\pstart=S_0,S_1,\ldots,S_n=\pend$ from $\pstart$ to $\pend$ with increments chosen from~$\nu$,
conditioned to remain in the quadrant $\Z_{\geq 0}^2$.
Then the law of $S_{\lfloor nt \rfloor} /\sqrt{n}$ converges weakly w.r.t.\ the $L^\infty$ norm on $[0,1]$ to that of a Brownian excursion (with diffusion matrix given by the second moments of $\nu$) into the nonnegative quadrant, starting and ending at the origin, with unit time length.
\end{prop}

In fact in this statement we do not need $\nu$ to have moments of all orders;
it suffices that $|\cdot|^\alpha$ has $\nu$-finite expectation, for a positive constant $\alpha$ defined in \cite{DurajWachtel}. The constant $\alpha$ depends on the angle of the cone $L (\R_{\geq 0}^2)$, where $L: \R^2 \to \R^2$ is a linear map for which $L(S_n)$ scales to a constant multiple of standard two-dimensional Brownian motion. In the setting of Theorems~\ref{thm::triangulationscalinglimit} and~\ref{thm:weighted-maps} below, $L$ can be the map that rescales the $(1,-1)$ direction by $1/\sqrt{3}$ and fixes the $(1,1)$ direction. In this case, the cone angle is $\pi/3$ and $\alpha = 3$.

The correlated Brownian excursion $(X,Y)$ in $\R_{\geq 0}^2$ referred to in the statement of Proposition~\ref{prop::constrainedbrownianbridge} is characterized by the Gibbs resampling property, which states that the following is true.  For any $0 < s < t < 1$, the conditional law of $(X,Y)$ in $[s,t]$ given its values in $[0,s]$ and $[t,1]$ is that of a correlated Brownian motion of time length $t-s$ starting from $(X(s),Y(s))$ and finishing at $(X(t),Y(t))$ conditioned on the positive probability event that it stays in $\R_{\geq 0}^2$.  The existence of this process follows from the results of \cite{shimura:cone}; see also \cite{Garbit}.

Now let us return to the study of random bipolar-oriented planar triangulations.
By Theorem~\ref{general} these correspond to paths in the nonnegative quadrant from the $y$-axis to the $x$-axis
which have increments of $(1,-1)$ and $(0,1)$ and $(-1,0)$.
Fix the boundary lengths $m+1$ and $n+1$, that is, fix the start $(0,m)$ and end $(n,0)$ of the walk,
and let the length $\ell$ get large.
Note that if $\nu$ is the uniform measure on the three values $(1,-1)$ and $(0,1)$ and $(-1,0)$, then the $\nu$-expectation of an increment $(X, Y)$ of the (unconstrained) walk is $(0,0)$.
Furthermore, (in the unconstrained walk) the variance of $X-Y$ is $2$ while the variance of $X+Y$ is $2/3$,
and the covariance of $X-Y$ and $X+Y$ is zero by symmetry.
Thus the variance in the $(1,-1)$ direction is $3$ times the variance in the $(1,1)$ direction.
The scaling limit of the random walk will thus be a Brownian motion with the corresponding covariance structure. We can summarize this information as follows:

\begin{theorem} \label{thm::triangulationscalinglimit}
Consider a uniformly random bipolar-oriented triangulation, sketched in the manner of Figure~\ref{fig::ao},
with fixed boundary lengths $m+1$ and $n+1$
and with the total number of edges given by $\ell$.
Let $S_0,S_1,\ldots,S_{\ell-1}$ be the corresponding lattice walk.
Then $S_{\lfloor \ell t \rfloor} / \sqrt{\ell}$ converges in law (weakly w.r.t.\ the $L^\infty$ norm on $[0,1]$), as $\ell \to \infty$ with $\ell\equiv -m-n+1\bmod 3$, to the Brownian excursion in the nonnegative quadrant starting and ending at the origin, with covariance matrix
$\begin{pmatrix}\sfrac{2}{3}&-\sfrac{1}{3}\\-\sfrac{1}{3}&\sfrac{2}{3}\end{pmatrix}$.
(This is the covariance matrix such that if the Brownian motion were unconstrained, the difference and sum of the two coordinates at time $1$ would be independent with respective variances $2$ and $2/3$.)
\end{theorem}

In particular, Theorem~\ref{thm::triangulationscalinglimit} holds when
the lattice path starts and ends at the origin, so that the left and
right sides of the planar map each have length $1$.  In this case, the
two sides can be glued together and treated as a single edge in the
sphere, and Theorem~\ref{thm::triangulationscalinglimit} can be
understood as a statement about bipolar maps on the sphere with a
distinguished south to north pole edge.

\medskip\medskip
Next we consider more general bipolar-oriented planar maps.  Suppose we allow not just triangles, but other face sizes.
Suppose that for nonnegative weights $a_2,a_3,\ldots$, we weight a bipolar-oriented planar map by $\prod_{k=2}^\infty a_k^{n_k}$ where $n_k$ is the number of faces with $k$ edges, and we use the convention $0^0=1$.  (Taking $a_k = 0$ means that faces with $k$ edges are forbidden.)  For maps with a given number of edges, this product is finite.  Then we pick a bipolar-oriented planar map with $\ell$ edges with probability proportional to its weight; the normalizing constant is finite, so this defines a probability measure if at least one bipolar map has positive weight.

To ensure that such bipolar maps exist, there is a congruence-type condition involving the number of edges $\ell$ and the set of face sizes $k$ with positive weight $a_k$.  We also use an analytic condition on the set of weights $a_k$ to ensure that random bipolar maps are not concentrated on maps dominated by small numbers of large faces.  When both these conditions are met, we obtain the limiting behavior as $\ell\to\infty$.

\enlargethispage{12pt}
\begin{theorem} \label{thm:weighted-maps}
Suppose that nonnegative face weights $a_2,a_3,\ldots$ are given,
and $a_k>0$ for at least one $k\geq 3$.
Let
\begin{equation} \label{eq:period}
 b = \gcd\big(\{k\geq 1: a_{2k}>0 \} \cup \{2k+1\geq 3 : a_{2k+1}>0\}\big)\,.
\end{equation}
Consider a bipolar-oriented planar map with
fixed boundary lengths $m+1$ and $n+1$
and with the total number of edges given by $\ell$,
chosen with probability proportional to the product of the face weights.
If $m+n$ is odd and all face sizes are even, or if
\begin{equation} \label{eq:cong-mod-b}
2\times(\ell-1)\equiv m+n\bmod b\,,
\end{equation}
does not hold,
then there are no such maps; otherwise, for $\ell$ large enough there are such maps.
Let $S_0,S_1,\ldots,S_{\ell-1}$ be the corresponding lattice walk.

Suppose $\sum_k a_k z^k$ has a positive radius of convergence $R$, and
\begin{equation} \label{eq:crosses}
 1 \leq \sum_{k=2}^\infty \frac{(k-1)(k-2)}{2} a_k R^{k}\,.
\end{equation}
Then for some finite $\lambda$ with $0<\lambda\leq R$
\begin{equation} \label{eq:drift0}
1 = \sum_{k=2}^\infty \frac{(k-1)(k-2)}{2} a_k \lambda^{k}\,.
\end{equation}
Suppose further $\lambda<R$, or $\lambda=R$ but also $\sum_k k^4 a_k R^k < \infty$.
Then as $\ell \to \infty$ while satisfying \eqref{eq:cong-mod-b},
the scaled walk $S_{\lfloor \ell t \rfloor} / \sqrt{\ell}$ converges in law (weakly w.r.t.\ the $L^\infty$ norm on $[0,1]$), to the Brownian excursion in the nonnegative quadrant starting and ending at the origin, with covariance matrix
that is a scalar multiple of $\begin{pmatrix}\sfrac{2}{3}&-\sfrac{1}{3}\\-\sfrac{1}{3}&\sfrac{2}{3}\end{pmatrix}$.

Furthermore, the walk is locally approximately i.i.d.:
For any $\eps_1>0$ there is an $\eps_2>0$ so that as $\ell\to\infty$, for any sequence of $\eps_2\ell$ consecutive moves that is disjoint from the first or last $\eps_1\ell$ moves, the $\eps_2\ell$ moves are within total variation distance $\eps_1$ from an i.i.d.\ sequence, in which move $m_{i,j}$ occurs with probability $a_{i+j+2} \lambda^{i+j}/C$ and move $m_e$ occurs with probability $\lambda^{-2}/C$, and $C$ is a normalizing constant.
\end{theorem}

\begin{remark}
  The constraint \eqref{eq:crosses} is to ensure that \eqref{eq:drift0} can be satisfied,
  which will imply that the lattice walk has a limiting step distribution that has zero drift.
  The next constraint implies that the limiting step distribution has finite third moment.
  When the weights $a_2,a_3,\ldots$ do not satisfy these constraints, the random walk excursion does not
  in general converge to a Brownian motion excursion.  Can one
  characterize bipolar-oriented planar maps in these cases?
 Can the inequality $\sum_k k^4 a_k \lambda^k < \infty$ be
   replaced with $\sum_k k^3 a_k \lambda^k < \infty$ (finite second moment for the step distribution)?
\end{remark}

\begin{remark}
  Theorem~\ref{thm:weighted-maps} applies to triangulations
  (giving Theorem~\ref{thm::triangulationscalinglimit} except for the
  scalar multiple in the covariance matrix),
  quadrangulations, or $k$-angulations for any fixed $k$,
  or more generally when one allows only a finite set of face sizes.
  The bound~\eqref{eq:crosses} is trivially satisfied in these cases
  since the radius of convergence is $R=\infty$.
\end{remark}

\begin{remark}
In the case where $1=a_2=a_3=\cdots$, i.e., the uniform distribution
on bipolar-oriented planar maps, the radius of convergence is $R=1$,
and $\lambda=1/2$, so Theorem~\ref{thm:weighted-maps} applies.
The step distribution $\nu$ of the walk is
\[\nu \{(-i,j) \} = \begin{cases}
2^{-i-j-3} & i,j \geq 0\,\,\,\, \mathrm{ or}\,\,\,\, i=j= -1 \\ 0 & \mathrm{otherwise.}
\end{cases}\]
In this case it is also possible to derive the distribution $\nu$ for uniformly random bipolar-oriented planar maps using a different bijection, one to noncrossing triples of lattice paths \cite{FPS}.
\end{remark}

\begin{remark}
Under the hypotheses of Theorem~\ref{thm:weighted-maps}, with $p_k$ defined as in its proof,
in a large random map a randomly chosen face has degree $k$ with limiting probability
\[\mathbb{P}(\text{face has degree $k$}) \to \frac{(k-1)p_k}{1-p_0}.\]
\end{remark}

\begin{proof}[Proof of Theorem~\ref{thm:weighted-maps}]
Since the right-hand side of \eqref{eq:drift0} increases monotonically from $0$ and is continuous on $[0,R)$, \eqref{eq:crosses} implies the existence of a solution $\lambda\in(0,R]$ to \eqref{eq:drift0}.
Since $a_k>0$ for some $k\geq 3$, $\lambda<\infty$.

Next let $a_0=1$ and define
\[ C = \frac{a_0}{\lambda^2} + \sum_{k=2}^\infty (k-1) a_k \lambda^{k-2}\,,
\]
which by our hypotheses is finite, and define
\[
p_k = \frac{a_k \lambda^{k-2}}{C}\,.
\]
Then the $p_k$'s define a random walk $(X_t,Y_t)$ in $\Z^2$, which assigns probabilities $p_0$ and  $p_{i+j}$ to steps $m_e$ and $m_{i,j}$ respectively (recall that there are $k-1$ possible steps of type $m_{i,j}$ where $i+j=k-2$, corresponding to a $k$-gon).

If we pick a random walk of length $\ell-1$ from $z_{\text{start}}$ to $z_{\text{end}}$
weighted by the $a_k$'s, it has precisely the same distribution as it
would have if we weighted it by the $p_k$'s instead,
because the total exponent of $\lambda$ for a walk from $z_{\text{start}}$ to $z_{\text{end}}$ is
$y_{\text{end}}-x_{\text{end}}-y_{\text{start}}+x_{\text{start}}$,
and because the total exponent of $C$ is $\ell-1$.
The advantage
of working with the $p_k$'s rather than the $a_k$'s is that they define a random walk,
which, as we verify next, has zero drift.

The drift of $X_t+Y_t$ is zero by symmetry.
The drift of $X_t-Y_t$ is
\be\label{zerodrift}
2 p_0 - \sum_{k\geq 2} (k-2)(k-1) p_k\,
\ee
which is zero by the definition of $p_k$ and \eqref{eq:drift0}.

Next we determine the period of the walk in $\Z^2$.
Consider the antidiagonal direction $Y_t-X_t$.  A move of type $m_e$ decreases this by $2$, and a move of type $m_{i,j}$, corresponding to a $k$-gon with $k=i+j+2$, increases it by $k-2$.  For even $k$, a move of type $m_{k/2-1,k/2-1}$ followed by $k/2-1$ moves of type $m_e$ returns the walk to its start after $k/2$ total moves.  For odd $k$, a move of type $m_{k-2,0}$ and a move of type $m_{0,k-2}$ followed by $k-2$ moves of type $m_e$ returns the walk to its start after $k$ total moves.  So we see that the period of the walk is no larger than $b$ as defined in \eqref{eq:period}.  If the period were smaller, then we could consider a minimal nonempty set of $t$ moves for which $Y_t-X_t=Y_0-X_0$ and $b\nmid t$.  Such a minimal set would contain no $k$-gon moves for even $k$ (since we could remove a $k$-gon move and $k/2-1$ type $m_e$ moves to get a smaller set), and at most one $k$-gon move for any given odd $k$
(since for odd $k$ we can remove $k$-gon moves in pairs along with $k-2$ type $m_e$ moves to get a smaller set).  Let $k_1,\dots,k_r$ be these odd $k$'s.  There are $(k_1+\cdots+k_r-2r)/2$ $m_e$ moves, for a total of $(k_1+\cdots+k_r)/2$ moves.
Then $2t=k_1+\cdots+k_r$, and since $b\mid k_1,\dots,b\mid k_r$,
we have $b \mid 2t$.  Since $r\geq1$, $b\mid k_1$, so $b$ is odd, and so in fact $b\mid t$.  Hence both the walk $(X_t,Y_t)$ and its projection $Y_t-X_t$ are periodic with period $b$.

For a face of size $k$, let $q(k)=k/2$ if $k$ is even and $q(k)=k$ if $k$ is odd.
The period $b$ is an integer linear combination of finitely many terms $q(k_1)<\cdots<q(k_s)$ where $a_{k_i}>0$.
We claim that
any multiple of $b$ which is at least $(s-1) q(k_s)^2$ is a nonnegative-integer linear combination of $q(k_1),\ldots,q(k_s)$.
To see this, let $c$ be a multiple of $b$ that is at least $(s-1) q(k_s)^2$.  We may write $c=\sum_{i=1}^s \beta_i q(k_i)$ where $\beta_i\in\Z$;
suppose that we choose the coefficients $\beta_1,\ldots,\beta_s$ to maximize the sum of the negative coefficients.  If some coefficient $\beta_i$
is negative, then there is another coefficient $\beta_j$ for which $\beta_j q(k_j) \geq q(k_s)^2 > q(k_i) q(k_j)$, in which case we could decrease
$\beta_j$ by $q(k_i)$ and increase $\beta_i$ by $q(k_j)$ to increase the sum of the negative coefficients.  This completes the proof of the claim.
Thus the walk in $\Z^2$ (not confined to the quadrant) may return to its start after any sufficiently large multiple of $b$ steps.

Suppose a walk in $\Z^2$ starts at $(0,m)$ and goes to $(n,0)$ after $t=\ell-1$ steps.
Consider the walk's projection in the antidiagonal direction: $(Y_t-X_t)-(Y_0-X_0)=-m-n$.  If $m+n$ is even, then the projected walk can reach its destination
after $(m+n)/2$ $m_e$ moves, and since $b$ is the period, it follows that $\ell-1\equiv (m+n)/2\bmod b$.  If $m+n$ is odd, then for the walk to exist there must be some odd $k$ with $a_k>0$.  The projected walk can reach its destination after an $m_{k-2,0}$ move and $(m+n+k-2)/2$ $m_e$ moves, and since $b$ is the period of the walk, $\ell-1\equiv (m+n+k)/2 \bmod b$.  In either case, the existence of such a walk implies $2(\ell-1) \equiv m+n \bmod b$.

If there are only even face sizes and $m+n$ is odd, there are no walks from $(0,m)$ to $(n,0)$.
Otherwise,
whether $m+n$ is even or there is an odd face size,
we can first choose face moves to change the $X_t+Y_t$ coordinate from $m$ to $n$, and then follow them by some number of $m_e$ moves to change the $Y_t-X_t$ coordinate to $-n$.
We may then follow these moves by a path from $(n,0)$ to itself with length given by any sufficiently large multiple of $b$.
Thus, for any
sufficiently large $\ell$ with $2\,(\ell-1) \equiv m+n \bmod b$, there is a walk within $\Z^2$ (not confined to the quadrant) from $(0,m)$ to $(n,0)$.

Next pick a face size $k\geq 3$ for which $a_k>0$.  For $s\geq 0$, the above walk in $\Z^2$ from $(0,m)$ to $(n,0)$ can be prepended with $(m_{0,k-2}^2 m_e^{k-2})^s$ and postpended with $(m_e^{k-2} m_{k-2,0}^2)^s$, and it will still go from $(0,m)$ to $(n,0)$.  For some sufficiently large~$s$, the walk will not only remain in the quadrant but will also travel arbitrarily far from the boundary of the quadrant, which gives the paths required by Proposition~\ref{prop::constrainedbrownianbridge}.

The variances of $X-Y$ and $X+Y$ are respectively
\begin{equation} \label{var[x-y]}
\Var[X-Y] = 4p_0 +\sum_{k\geq 2}(k-2)^2(k-1) p_k,
\end{equation}
and
\begin{equation}\label{var[x+y]}
\Var[X+Y] = \sum_{k\geq 2} p_k ((k-2)^2+(k-4)^2+\cdots+(-k+2)^2)= \sum_{k\geq 2} p_k \times 2\binom{k}{3},
\end{equation}
which are both positive and finite by our hypotheses.
Using the zero-drift condition~\eqref{zerodrift}, we may combine \eqref{var[x-y]} and \eqref{var[x+y]} to obtain
\[
\Var[X-Y] = \sum_{k\geq 2}(k-2)(k-1)k \,p_k = \sum_{k\geq 2} p_k \,6 \binom{k}{3} = 3 \Var[X+Y].
\]

Then we apply Proposition~\ref{prop::constrainedbrownianbridge}.
Since the ratio of variances is $3$, we need the walk's step distribution to have a finite third moment (see the comments after Proposition~\ref{prop::constrainedbrownianbridge}).
Since there are $|k-1|$ steps of type~$p_k$, the third moment of the step distribution is finite when
\[ \sum_{k\geq 2} p_k k^4 = \frac{1}{C}\sum_{k\geq 2} a_k \lambda^{k-2} k^4 < \infty
\]
which is implied by our hypotheses.
Hence by Proposition~\ref{prop::constrainedbrownianbridge} the scaling limit of the walk is a correlated Brownian excursion in the quadrant.

The local approximate i.i.d.\ nature of the walk follows from an entropy maximization argument
together with the facts that we showed above.
\end{proof}

\begin{remark} \label{rem::otherkappavariances} If one relaxes the
  requirement that the probabilities assigned by the step distribution
  $\nu$ be the same for
  all increments corresponding to a given face size, one can find a $\nu$
  such that the
  expectation is still $(0,0)$ and when $(X,Y)$ is sampled from~$\nu$,
  the law is still symmetric w.r.t.\ reflection about the line $y=-x$
  but the variance ratio $\Var[X-Y] / \Var[X+Y]$ assumes any value
  strictly between $1$ and~$\infty$.  Indeed, one approaches one
  extreme by letting $(X,Y)$ be (close to being) supported on the
  $y=-x$ antidiagonal, and the other extreme by letting $(X,Y)$ be
  (close to being) supported on the $x$- and $y$-axes far from the
  origin (together with the point $(1,-1)$).  The former corresponds
  to a preference for nearly balanced faces (in terms of the number of
  clockwise and counterclockwise oriented edges) while the latter
  corresponds to a preference for unbalanced faces.
\end{remark}

\begin{remark}\label{rem::infinitevolumelimit}
  In each of the models treated above, it is natural to consider an
  ``infinite-volume limit'' in which lattice path increments indexed
  by $\Z$ are chosen i.i.d.\ from~$\nu$. The standard central limit
  theorem then implies that the walks have scaling limits given by a
  Brownian motion with the appropriate covariance matrix.
\end{remark}

\section{Bipolar-oriented triangulations} \label{sec::triangle}

\subsection{Enumeration}

The following corollary is an easy consequence of the bijection.
The formula itself goes back to Tutte \cite{tutte:triangulations};
Bousquet-Melou gave another proof together with a discussion of the
bipolar orientation interpretation \cite[Prop.~5.3, eqn.~(5.11) with $j=2$]{bousquet-melou:maps}.

\begin{corollary} \label{cor:enumerate}
The number of bipolar-oriented triangulations of the sphere with $\ell$ edges in which S and N are adjacent and marked
is (with $\ell=3n$)
\[B_{\ell} =\frac{2\,(3n)!}{(n+2)!\,(n+1)!\,n!}\]
(and zero if $\ell$ is not a multiple of $3$).
\end{corollary}

\begin{proof}
  In a triangulation $2E=3F$ so the number of edges is a multiple of
  $3$.  Since S and N are adjacent, there is a unique embedding in the
  disk
  so that the west boundary has
  length~$1$
  and the east boundary has length~$2$.
  The lattice walks as discussed there go from $(0,0)$ to $(1,0)$.  It
  is convenient to concatenate the walk with a final $m_{1,0}$ step,
  so that the walks are from $(0,0)$ to $(0,0)$ of length $\ell$ and
  remain in the first quadrant.  Applying a shear
  $\begin{pmatrix}1&1\\0&1\end{pmatrix}$, the walks with steps
  $m_e,m_{0,1},m_{1,0}$ become walks with steps $(1,0),(0,1),(-1,-1)$
  which remain in the domain $y\ge x\ge0$.  Equivalently this is the
  number of walks from $(0,0,0)$ to $(n,n,n)$ with steps
  $(1,0,0),(0,1,0),(0,0,1)$ remaining in the domain $y\ge x\geq z$.
  These are the so-called 3D Catalan numbers, see
  \href{https://oeis.org/A005789}{A005789} in the OEIS.
\end{proof}

\subsection{Vertex degree}

Using the bijection between paths and bipolar-oriented maps, we can
easily get the distribution of vertex degrees of a large
bipolar-oriented triangulation.

\begin{prop}
  In a large bipolar-oriented planar triangulation with fixed boundary lengths $m+1$ and $n+1$, as the number of edges $\ell$ tends to $\infty$
 with $\ell+m+n\equiv 1\bmod 3$, the limiting
  in-degree and out-degree distributions of a random vertex are
  independent and geometrically distributed (starting at $1$)
  with mean $3$.
\end{prop}

\begin{proof}
We examine the construction of bipolar-oriented planar maps
when the steps give triangles.  Any new vertex or new edge
is adjoined to the marked bipolar map on its eastern boundary,
which we also call the \textit{frontier}.

A new vertex is created by an $m_{0,1}$ move, or an $m_e$ move if there are currently
no frontier vertices above the active vertex, and when a vertex is created it is the active vertex.
Each subsequent move moves frontier vertices relative to the active vertex, so let us record
their position with respect to the active vertex by integers, with positive integers recording the position
below the active vertex and negative integers recording the position above it.
See Figure \ref{3ms}.
\begin{figure}[htbp]
\begin{center}
\includegraphics[width=5in]{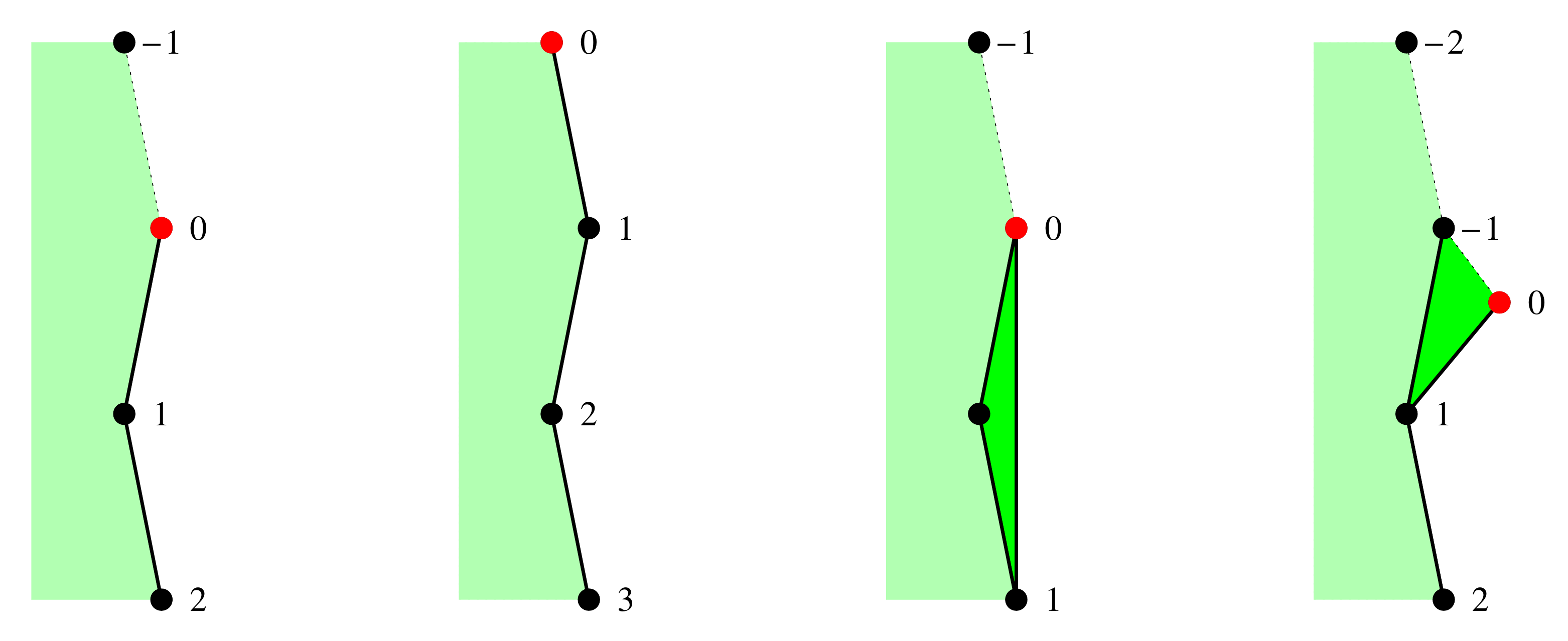}
\newcommand\capsp{\hspace{1.39in}}
\mbox{\small \clap{frontier before move}\capsp\clap{after $m_e$}\capsp\clap{after $m_{1,0}$}\capsp\clap{after $m_{0,1}$}\ \ \ }
\end{center}
\caption{\label{3ms}  Action of the three moves $m_e,m_{1,0},m_{0,1}$ on the frontier.  The vertex positions (relative to the active vertex) are shown.}
\end{figure}

The following facts are easily verified.
\begin{enumerate}
\item
A vertex moves off the frontier exactly when it is at position $1$ and an $m_{1,0}$ move takes place.
\item
$m_e$ moves increase the index of vertices by $1$.
\item
$m_{0,1}$ moves decrease the index of a vertex by $1$ if it is non-positive, else leave it fixed.
\item
$m_{1,0}$ moves decrease the index by $1$ if it is $\ge 2$, else leave it fixed (if the index is $1$ it is moved off of the frontier).
\item
Except for the start vertex of the initial structure, whenever a vertex
is created, its in-degree is $1$ and its out-degree is $0$.
\item \label{fact:in-out}
The in-degree of a vertex increases by $1$ each time it visits position $0$,
the out-degree increases each time it visits position $1$.
\end{enumerate}

The transition diagram is summarized here:

\begin{center}
\begin{tikzpicture}[->,shorten >=1pt,auto,node distance=1.8cm,
  thick,main node/.style={circle,fill=yellow,draw,font=\large\bfseries,minimum size=10mm}]

  \node[main node] (0) {$0$};
  \node[main node] (1) [right of=0] {$1$};
  \node[main node] (2) [right of=1] {$2$};
  \node[main node] (3) [right of=2] {$3$};
  \node (4) [right of=3] {$\cdots$};
  \node[main node] (-1) [left of=0] {$-1$};
  \node[main node] (-2) [left of=-1] {$-2$};
  \node[main node] (-3) [left of=-2] {$-3$};
  \node (-4) [left of=-3] {$\cdots$};
  \node (off) at (0,-1.5) {vertex off frontier};
  \node (born) at (1.4,2.3) {vertex created};
  \path[every node/.style={font=\sffamily\small,inner sep=1pt}]

  (-4)
  edge [bend left] node {$m_e$} (-3)
  (-3) edge [loop above] node {$m_{1,0}$} (-3)
  edge [bend left] node {$m_e$} (-2)
  edge [bend left] node {$m_{0,1}$} (-4)
  (-2) edge [loop above] node {$m_{1,0}$} (-2)
  edge [bend left] node {$m_e$} (-1)
  edge [bend left] node {$m_{0,1}$} (-3)
  (-1) edge [loop above] node {$m_{1,0}$} (-1)
  edge [bend left] node {$m_e$} (0)
  edge [bend left] node {$m_{0,1}$} (-2)
  (0) edge [loop above] node {$m_{1,0}$} (0)
  edge [bend left] node {$m_e$} (1)
  edge [bend left] node {$m_{0,1}$} (-1)
  (1) edge [loop above] node {$m_{0,1}$} (1)
  edge [bend left] node {$m_e$} (2)
  edge node {$m_{1,0}$} (off)
  (2) edge [loop above] node {$m_{0,1}$} (2)
  edge [bend left] node {$m_e$} (3)
  edge [bend left] node {$m_{1,0}$} (1)
  (3) edge [loop above] node {$m_{0,1}$} (3)
  edge [bend left] node {$m_e$} (4)
  edge [bend left] node {$m_{1,0}$} (2)
  (4)
  edge [bend left] node {$m_{1,0}$} (3)
  (born) edge (0)
  ;
\end{tikzpicture}
\end{center}

For the purposes of computing the final in-degree and out-degree
of a vertex, we can
simply count the number of visits to $0$ before its index becomes
positive, and then count the number of visits to $1$ before it is
absorbed in the interior of the structure.

Since $m$ and $n$ are held fixed as $\ell\to\infty$, almost all vertices in the bipolar map are created by $m_{0,1}$ moves.
By the local approximate i.i.d.\ property of the walk proved in Theorem~\ref{thm:weighted-maps}, we see that the moves in the transition diagram above converge weakly to a Markov chain where each transition occurs with probability $1/3$.

The Markov chain starts at $0$, and on each visit to $0$ there is a $1/3$ chance of going to $1$ and a $2/3$ chance of eventually returning to $0$.  On each visit to $1$, there is a $1/3$ chance of exiting and a $2/3$ chance of eventually returning to $1$.  In the Markov chain, the number of visits to $0$ and $1$ are a pair of independent geometric random variables with minimum $1$ and mean $3$, which in view of fact~\ref{fact:in-out} above, implies the proposition.
\end{proof}

\section{Scaling limit} \label{sec::peanosphere}

\enlargethispage{12pt}
\subsection{Statement}

The proof of the following theorem is
an easy computation upon application of the infinite-volume tree-mating theory introduced in \cite{DMS:mating}, a derivation of the relationship between the SLE/LQG parameters and a certain variance ratio in \cite{DMS:mating,GHMS:covariance},
and a finite volume elaboration in \cite{miller-sheffield:finite-trees}.
For clarity and motivational purposes we will reverse the standard conventions and
give the proof first, explaining the relevant background in the following subsection.

\begin{theorem}\label{scalinglimit}
The scaling limit of the bipolar-oriented planar map with its interface curve, with fixed boundary lengths $m+1$ and $n+1$, and number of edges $\ell\to\infty$ (with a possible congruence restriction on $\ell$, $m$, and $n$ to ensure such maps exist),
with respect to the peanosphere topology, is a $\sqrt{4/3}$-LQG sphere decorated by an independent
$\SLE_{12}$ curve.
\end{theorem}

We remark that the peanosphere topology is neither coarser nor finer than other natural topologies, including in particular those that we discuss in
the Section~\ref{peanosphere}.

\begin{proof}[Proof of Theorem~\ref{scalinglimit}]
In Section~\ref{pathscalinglimit} it was shown that
the interface function $(X_t,Y_t)$ for the bipolar-oriented random planar map
converges as $\ell\to\infty$ to a Brownian excursion in the nonnegative quadrant with increments $(X,Y)$ having
covariance matrix (up to scale)
$\begin{pmatrix}\sfrac{2}{3}&-\sfrac{1}{3}\\-\sfrac{1}{3}&\sfrac{2}{3}\end{pmatrix}$,
that is
$X-Y$ and $X+Y$ are independent, and $\text{Var}[X-Y] = 3\text{Var}[X+Y]$.

The fact that the limit is a Brownian excursion implies, by \cite[Theorem~1.13]{DMS:mating} and the finite volume variant in \cite{miller-sheffield:finite-trees} and \cite[Theorem~1.1]{GHMS:covariance}, that the scaling limit in the peanosphere topology is a peanosphere, that is, a $\gamma$-LQG sphere decorated by an independent space-filling $\SLE_{\kappa'}$.  The values $\gamma, \kappa'$ are determined by the covariance structure of the limiting Brownian excursion. The ratio of variances $\text{Var}[X-Y]/\text{Var}[X+Y]$ takes the form
\begin{equation} \label{eqn::cosratio} (1+\cos[4\pi/\kappa'])/(1-\cos[4\pi/\kappa']).\end{equation}
This relation was established for $\kappa\in[2,4)$ and $\kappa'\in (4,8]$ in \cite{DMS:mating},
and more generally for $\kappa \in (0,4)$ and $\kappa'\in(4,\infty)$ in \cite{GHMS:covariance}.\footnote{There is as yet no analogous construction corresponding to the limiting case $\kappa = \kappa'=4$, where~\eqref{eqn::cosratio} is zero so that $\Var(X-Y) = 0$ and $X=Y$ a.s.\ It is not clear what such a construction would look like, given that space-filling $\SLE_{\kappa'}$ has only been defined for $\kappa' > 4$, not for $\kappa'=4$, and the peanosphere construction in Section~\ref{peanosphere} is trivial when the limiting Brownian excursion is supported on the diagonal $x=y$.}
Setting it equal to $3$ and solving we find $\kappa' = 12$. For this value of $\kappa'$ we have
$\gamma = \sqrt{\kappa} = \sqrt{16/\kappa'} = \sqrt{4/3}$.
\end{proof}

\begin{remark} If the covariance ratios vary as in Remark~\ref{rem::otherkappavariances}, then the $\kappa'$ values varies between $8$ and $\infty$. In other words, one may obtain any $\kappa' \in (8,\infty)$, and corresponding $\gamma = \sqrt{16/\kappa'}$, by introducing weightings that favor faces more or less balanced.
\end{remark}

\begin{remark}
The infinite-volume variant described in Remark~\ref{rem::infinitevolumelimit} corresponds to the mated pair of \textit{infinite-diameter\/} trees first described in \cite{DMS:mating}, which in turn corresponds to the so-called $\gamma$-quantum cone described in the next subsection.
\end{remark}

\subsection{Peanosphere background}
\label{peanosphere}

\begin{figure}[ht!]
\begin{center}
\includegraphics[scale=0.85,page=4]{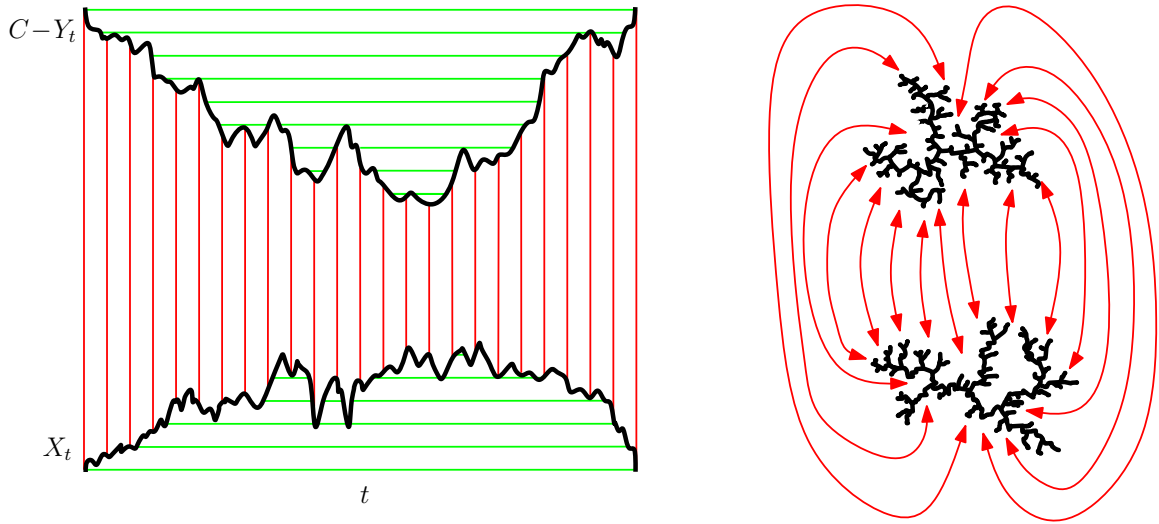}
\end{center}
\caption{\label{fig::lamination}  Gluing together a pair of CRTs to obtain a topological sphere.  Illustration of the peanosphere construction.  (This figure first appeared in \cite{DMS:mating}.)}
\end{figure}

The purpose of this section is to give a brief description of how Liouville quantum gravity (LQG) surfaces \cite{duplantier-sheffield:KPZ} decorated by independent $\SLE$ processes can be viewed as matings of random trees which are related to Aldous' continuum random tree (CRT) \cite{aldous:crt1,aldous:crt2,aldous:crt3}.  The results that underly this perspective are established in \cite{DMS:mating, miller-sheffield:finite-trees}, building on prior results from  \cite{duplantier-sheffield:KPZ,sheffield:welding,sheffield:cle-trees,miller-sheffield:ig1,miller-sheffield:ig2,miller-sheffield:ig3,miller-sheffield:ig4}.

Recall that if $h$ is an instance of the Gaussian free field (GFF)
on a planar domain~$D$ with zero-boundary conditions and $\gamma \in (0,2)$, then the $\gamma$-LQG surface associated with $h$ of parameter $\gamma$ is described by the measure $\mu_h$ on $D$ which formally has density $e^{\gamma h}$ with respect to Lebesgue measure.  As $h$ is a distribution and does not take values at points, this expression requires interpretation.  One can construct this measure rigorously by considering approximations $h_\varepsilon$ to $h$ (by averaging the field on circles of radius $\epsilon$)
 and then take $\mu_h$ to be the weak limit as $\varepsilon \to 0$ of $\varepsilon^{\gamma^2/4} e^{h_\varepsilon(z)} dz$ where $dz$ denotes Lebesgue measure on $D$; see \cite{duplantier-sheffield:KPZ}.  If one has two planar domains $D_1,D_2$, a conformal transformation $\varphi \colon D_1 \to D_2$, an instance of the GFF $h_2$ on $D_2$, and lets
\begin{equation}
\label{eqn::coordinate_change}
h_1 = h_2 \circ \varphi + Q \log| \varphi'| \quad\text{where}\quad Q = \frac{2}{\gamma} + \frac{\gamma}{2}
\end{equation}
then the $\gamma$-LQG measure $\mu_{h_2}$ associated with $h_2$ is a.s.\ the image under $\varphi$ of the $\gamma$-LQG measure $\mu_{h_1}$ associated with $h_1$.  A \textit{quantum surface\/} is an equivalence class of fields $h$ where we say that two fields are equivalent if they are related as in~\eqref{eqn::coordinate_change}.

This construction generalizes to any law on fields $h$ which is absolutely continuous with respect to the GFF.  The results in this article will be related to two such laws \cite{sheffield:welding,DMS:mating}:
\begin{enumerate}
\item The $\gamma$-quantum cone (an infinite-volume surface).
\item The $\gamma$-LQG sphere (a finite-volume surface).
\end{enumerate}
We explain how they can both be constructed with the ordinary GFF $h$ as the starting point.

The $\gamma$-quantum cone can be constructed by the following limiting procedure starting with an instance of the GFF $h$ as above.  Fix a constant $C > 0$ and note that adding $C$ to $h$ has the effect of multiplying areas as measured by $\mu$ by the factor $e^{\gamma C}$.  If one samples $z \in D$ according to $\mu$ and then rescales the domain so that the mass assigned by $\mu_{h+C}$ to $B(z,1)$ is equal to $1$ then the law one obtains in the $C \to \infty$ limit is that of a $\gamma$-quantum cone.  (The construction given in \cite{sheffield:welding,DMS:mating} is more direct in the sense that a precise recipe is given for sampling from the law of the limiting field.)  That is, a $\gamma$-quantum cone is the infinite-volume $\gamma$-LQG surface which describes the local behavior of an $\gamma$-LQG surface near a $\mu_h$-typical point.

The (unit area) $\gamma$-LQG sphere can also be constructed using a limiting procedure using the ordinary GFF $h$ as above as the starting point.  This construction works by first fixing $C > 0$ large, $\epsilon > 0$ small, and then \textit{conditioning\/} on the event that the amount of mass that $\mu$ assigns to $D$ is in $[e^{\gamma C},e^{\gamma(C+\epsilon)}]$,
so that the amount mass assigned to $D$ by $\mu_{h-C}$ is in $[1,e^{\gamma \epsilon}]$, then sends first $C \to \infty$ and then $\epsilon \to 0$.  (The constructions given in \cite{DMS:mating,miller-sheffield:finite-trees} are more direct because they involve precise recipes for sampling from the law of the limiting~$h$.)  One can visualize this construction by imagining that conditioning the area to be large (while keeping the boundary values of $h$ constrained to be $0$) leads to the formation of large a bubble.  In the $C \to \infty$ limit, the opening of the bubble (which is the boundary of the domain) collapses to a single point, and it turns out that this point is typical (i.e., conditioned on the rest of the surface its law is given by that of the associated $\gamma$-LQG measure).

In \cite{DMS:mating,miller-sheffield:finite-trees}, it is shown that it is possible to represent various types of $\gamma$-LQG surfaces (cones, spheres, and disks) decorated by an independent $\SLE$ as a gluing of a pair of continuous trees.  We first explain a version of this construction in which $\gamma=\sqrt{2}$ and the surface is a unit-area LQG sphere decorated with an independent $\SLE_8$.  Let $X$ and $Y$ be independent one-dimensional Brownian excursions parametrized by $[0,1]$.
Let $C$ be large enough so that the graphs of $X$ and $C-Y$ are disjoint, as illustrated in Figure~\ref{fig::lamination}.
We define an equivalence relation~$\sim$ on the rectangle $R = [0,1] \times [0,C]$ by declaring to be equivalent points which lie on either:
\begin{enumerate}
\item horizontal chords either entirely below the graph of $X$ or entirely above graph of $C-Y$ (green lines in Figure~\ref{fig::lamination}), or
\item vertical chords between the graphs of $X$ and $C-Y$ (red lines in Figure~\ref{fig::lamination}).
\end{enumerate}
We note that under $\sim$, all of $\partial R$ is equivalent so we may think of $\sim$ as an equivalence relation on the two-dimensional sphere~$\s^2$.  It is elementary to check using Moore's theorem \cite{moore:spheres} (as explained in \cite[Section~1.1]{DMS:mating}) that almost surely the topological structure associated with $R/\!\sim$ is homeomorphic to $\s^2$.  This sphere comes with additional structure, namely:
\begin{enumerate}
\item a space-filling path\footnote{As explained just below, $\eta'$ is related to an $\SLE_{\kappa'}$ curve with $\kappa' > 4$.  We use the convention here from \cite{miller-sheffield:ig1,miller-sheffield:ig2,miller-sheffield:ig3,miller-sheffield:ig4}, which is to use a prime whenever $\kappa' > 4$.} $\eta'$ (corresponding to the projection of the path which follows the red lines in Figure~\ref{fig::lamination} from left to right), and
\item a measure $\mu$ (corresponding to the projection of Lebesgue measure on $[0,1]$).
\end{enumerate}
We refer to this type of structure as a \textit{peanosphere}, as it is a topological sphere decorated with a path which is the peano curve associated with a space-filling tree.

The peanosphere associated with the pair $(X,Y)$ does not \textit{a priori\/} come with an embedding into the Euclidean sphere $\s^2$.  However, it is shown in \cite{DMS:mating, miller-sheffield:finite-trees} that there is a canonical embedding (up to M\"obius transformations)
of the peanosphere associated with $(X,Y)$ into $\s^2$, which is measurable with respect to $(X,Y)$.
This embedding equips the peanosphere with a conformal structure.  The image of $\mu$ under this embedding is a $\sqrt{2}$-LQG sphere, see \cite{DMS:mating, miller-sheffield:finite-trees} as well as \cite{DKRV:sphere,AHS:twoperspectives}), and the law of the space-filling path $\eta'$ is the following natural version of $\SLE_8$ in this context \cite{miller-sheffield:ig4}:  If we parametrize the $\sqrt{2}$-LQG sphere by the Riemann sphere $\wh{\C}$, then $\eta'$ is equal to the weak limit of the law of an $\SLE_8$ on $B(0,n)$ from $-i n$ to $i n$ with respect to the topology of local uniform convergence when parametrized by Lebesgue measure.  (The construction given in \cite{miller-sheffield:ig4} is different and is based on the GFF.)  The random path $\eta'$ and the random measure $\mu$ are coupled together in a simple way.  Namely, given $\mu$, one samples from the law of the path by first sampling an $\SLE_8$ (modulo time parametrization) independently of $\mu$ and then reparametrizing it according to $\mu$-area (so that in $t$ units of time it fills $t$ units of $\mu$-area).

This construction generalizes to all values of $\kappa' \in (4,\infty)$.  In the more general setting, we have that $\gamma = \sqrt{\kappa}$ where $\kappa = 16/\kappa' \in (0,4)$, and the pair of independent Brownian excursions is replaced with a continuous process $(X,Y)$ from $[0,1]$ into $\R_{\geq 0}^2$ which is given by the linear image of a two-dimensional Brownian excursion from the origin to the origin in the Euclidean wedge of opening angle
\[ \theta = \frac{\pi \gamma^2}{4} = \frac{\pi \kappa}{4} = \frac{4\pi}{\kappa'}\,\]
see \cite{DMS:mating, miller-sheffield:finite-trees, GHMS:covariance}.
(In the infinite-volume version of the peanosphere construction, the Brownian excursions $(X,Y)$ are replaced with Brownian motions, and the corresponding underlying quantum surface is a $\gamma$-quantum cone~\cite{DMS:mating}.)

The main results of \cite{DMS:mating, miller-sheffield:finite-trees} imply that the information contained in the pair $(X,Y)$ is a.s.\ \textit{equivalent\/} to that of the associated $\SLE_{\kappa'}$-decorated $\gamma$-LQG surface.  More precisely, the map $f$ from $\SLE_{\kappa'}$-decorated $\gamma$-LQG surfaces to Brownian excursions is almost everywhere well-defined and almost everywhere invertible, and both $f$ and $f^{-1}$ are measurable.

The peanosphere construction leads to a natural topology on surfaces which can be represented as a gluing of a pair of trees (a space-filling tree and a dual tree), as illustrated in Figure~\ref{fig::lamination}.  Namely, such a tree-decorated surface is encoded by a pair of continuous functions $(X,Y)$ where $X$ (resp.\ $Y$) is given by the interface function of the tree (resp.\ dual tree) on the surface.  We recall that the interface function records the distance of a point on the tree to the root when one traces its boundary with unit speed.  We emphasize that both continuum and discrete tree-decorated surfaces can be described in this way.  In the case of a planar map, we view each edge as a copy of the unit interval and use this to define ``speed.''  Equivalently, one can consider the discrete-time interface function and then extend it to the continuum using piecewise linear interpolation.  Applying a rescaling to the planar map corresponds to applying a rescaling to the discrete pair of trees, hence their interface functions.  If we have two tree-decorated surfaces with associated pairs of interface functions $(X,Y)$ and $(X',Y')$, then we define the distance between the two surfaces simply to be the sup-norm distance between $(X,Y)$ and $(X',Y')$.  That is, the peanosphere topology is the restriction of the sup-norm metric to the space of pairs of continuous functions which arise as the interface functions associated with tree-decorated surfaces.

The peanosphere approach to SLE/LQG convergence (i.e., identifying a natural pair of trees in the discrete model and proving convergence in the topology where two configurations are close if their tree interface functions are close) was introduced in \cite{sheffield:burgers,DMS:mating} to deal with infinite-volume limits of FK-cluster-decorated random planar maps, which correspond to $\kappa \in [2,4)$ and $\kappa' \in (4,8]$. Extensions to the finite volume case and a ``loop structure'' topology appear in \cite{GMS:cone_times,gwynne-sun:finite_volume,gwynne-sun:finite, gwynne-miller:topology}.

Since bipolar-oriented planar maps converge in the peanosphere topology to $\SLE_{12}$-decorated $\sqrt{4/3}$-LQG,
we conjecture that they also converge in other natural topologies, such as
\begin{itemize}
\item The \textit{conformal path topology\/} defined as follows. Assume we have selected a method of ``conformally embedding'' discrete planar maps in the sphere. (This might involve circle packing, Riemann uniformization, or some other method.) Then the green path in Figure~\ref{fig::ao} becomes an actual path: a function $\eta_n$ from $[0,1]$ to the unit sphere (where $n$ is the number of lattice steps) parameterized so that at time $k/n$ the path finishes traversing its $k$th edge. An $\SLE_{12}$-decorated $\sqrt{4/3}$-LQG sphere can be described similarly by letting $\eta$ be the $\SLE$ path parameterized so that a $t$ fraction of LQG volume is traversed between times $0$ and $t$. (Note that the parameterized path $\eta$ encodes both the LQG measure \textit{and\/} the $\SLE$ path.) The conformal path topology is the uniform topology on the set of paths from $[0,1]$ to the sphere. The conjecture is that $\eta_n$ converges to $\eta$ weakly w.r.t.\ the uniform topology on paths. See \cite{duplantier-sheffield:KPZ, sheffield:welding} for other conjectures of this type.
\item The Gromov--Hausdorff--Prokhorov--uniform topology on metric measure spaces decorated with a curve.  So far, convergence in this topology has only proved in the setting of a uniformly random planar map decorated by a self-avoiding walk (SAW) to $\SLE_{8/3}$ on $\sqrt{8/3}$-LQG in \cite{gwynne-miller:saw,gwynne-miller:gluing,gwynne-miller:uihpq}.  These works use as input the convergence of uniformly random planar maps to the Brownian map \cite{legall:bm,miermont:bm} and the construction of the metric space structure of $\sqrt{8/3}$-LQG \cite{miller-sheffield:qle,miller-sheffield:axiomatic,miller-sheffield:metric,miller-sheffield:continuity,miller-sheffield:determined,miller-sheffield:finite-trees}.  It is still an open problem to endow $\gamma$-LQG with a canonical metric space structure for $\gamma \neq \sqrt{8/3}$ and to prove this type of convergence result for random planar maps with other models from statistical physics.
\end{itemize}

An interesting problem which illustrates some of the convergence issues that arise is the following: In the discrete setting, the interface functions between the NW and SE trees determine the bipolar map which in turn determine the interface functions between the NE and SW trees.  Likewise, in the continuous setting, the interface functions (a Brownian excursion) between the NW and SE trees a.s.\ determine the SLE-decorated LQG which in turn a.s.\ determine the interface function (another Brownian excursion) between the NE and SW trees.
\begin{conjecture}
The joint law of both NW/SE and NE/SW interface functions of a random bipolar-oriented planar map converges to the joint law of both NW/SE and NE/SW interface functions of $\SLE_{12}$-decorated $\sqrt{4/3}$-LQG.
\end{conjecture}
One might expect to be able to approximate the discrete NW/SE interface function with a continuous function, obtain the corresponding continuous NE/SW function, and hope that this approximates the discrete NE/SW function.
One problem with this approach is that while the maps $f^{-1}$ and $f$ are measurable, they are (presumably almost everywhere) discontinuous,
so that even if two interface functions are close, it does not follow that the corresponding measures and paths are close.
However, since Brownian excursions are random perturbations rather than ``worst case'' perturbations of random walk excursions,
we expect the joint laws to converge despite the discontinuities of $f$ and $f^{-1}$.

\section{Imaginary geometry: why \texorpdfstring{$\kappa' = 12$}{\textkappa'=12} is special}\label{imaginarysection}
When proving that a family of discrete random curves has $\SLE_\kappa$ as a scaling limit, it is sometimes possible to figure out in advance what $\kappa$ should be by proving that there is only one $\kappa$ for which $\SLE_\kappa$ has some special symmetry. For example, it is by now well known that $\SLE_6$ is the only $\SLE$ curve with a certain \textit{locality\/} property (expected of any critical percolation interface scaling limit) and
that
$\SLE_{8/3}$ is the only $\SLE$ curve with with a certain \textit{restriction\/} property (expected of any self-avoiding-walk scaling limit).
The purpose of this section is to use the imaginary geometry theory of \cite{miller-sheffield:ig1,miller-sheffield:ig4} to explain what is special about the values $\kappa=4/3$ and $\kappa'=12$.

\subsection{Winding height gap for uniform spanning trees}
The connections between winding height functions, statistical
mechanics models, and height gaps are nicely illustrated in the
discrete setting by the uniform spanning tree (UST).  Temperley showed
that spanning trees on the square grid are in bijective correspondence
with dimer configurations (perfect matchings) on a larger square grid.  Dimer configurations have a height function which is
known to converge to the Gaussian free field \cite{dominoconformal}.
Under the Temperley correspondence, this dimer height function is related to the ``winding'' of the spanning tree, where the winding of a given edge in the tree is defined to be the number of right turns minus the number of left turns taken by the tree path from that edge to the root \cite{KPW}.  If we multiply the dimer
heights by $\pi/2$,  then this function describes the accumulated amount of angle by which the path has turned right on its journey toward the root.

Notice that if $v$ is a vertex along a branch of a spanning tree, and there are two tree edges off of that branch that merge into the vertex $v$ from opposite directions, then the winding height at the
edge just to the right of the branch is larger by $\pi$ than at
the edge just to the left, see Figure~\ref{heightgap-ust}.  Because of this, it is intuitively natural to expect that there will typically be a ``winding height gap'' across the long tree branch of magnitude $\pi$, i.e., the winding just right of the long tree branch should be (on average) larger by $\pi$ than the winding just left of the tree branch.

\begin {figure}[ht!]
\begin {center}
\includegraphics[scale=.9]{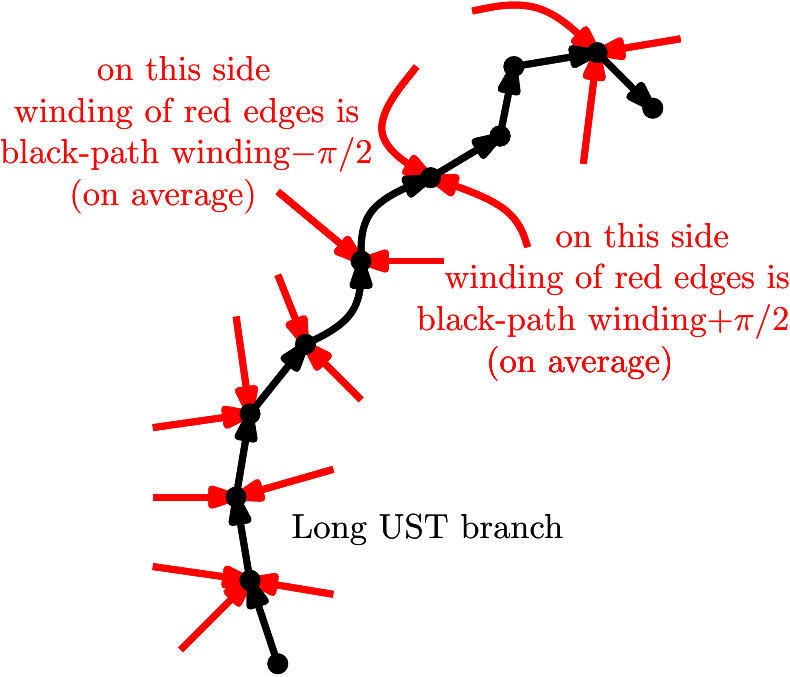}
\caption {\label{heightgap-ust} Every edge in a uniform spanning tree is assigned a ``winding'' --- a real number that indicates the total amount of right turning (minus left turning) that takes place as one moves along the
tree
from the midpoint of that edge to the root.  The black directed edges are those of a long branch in the UST directed towards the root, and adjacent spanning tree edges are shown in red.
Given the long black path, the conditional law of the configuration on the left side of the path does not depend on the orientation of the long black path.  Thus, by symmetry, one expects (on average) a $\pi/2$ angle gap between red edges and their black neighbors.  This accounts for a total ``average winding gap'' of $\pi$ between the left and right side.}
\end {center}
\end {figure}

\subsection{Winding height gap for SLE}
Imaginary geometry extends these notions of winding height function and the height gap to SLE.
The starting point is an instance $h$ of the GFF, which we divide by a parameter~$\chi>0$ to convert
into units of radians.  $\SLE_\kappa$ can be constructed as a flow line of the vector field
in which $z$ is assigned the complex unit vector $e^{i h/\chi}$,
where
\[\chi = \frac{2}{\sqrt{\kappa}} - \frac{\sqrt{\kappa}}{2}\,.\]
Although this vector field does not make literal sense, as $h$ is a distribution and not a function, one can still construct the flow lines in a natural way \cite{miller-sheffield:ig1,miller-sheffield:ig4}.
While it has been conjectured that these GFF flow lines are
limits of flow lines of mollified versions of the GFF, their rigorous construction follows a
different route. One first proves that they are the unique paths
coupled with the GFF that satisfy certain axiomatic properties
(regarding the conditional law of the field given the path), and then
establishes \textit{a posteriori\/} that the paths are uniquely determined by the GFF.

We interpret flow lines of $e^{i h/\chi}$ as ``east going'' rays in an ``imaginary geometry''.  A ray of a different angle $\theta$ is a flow line of $e^{ i [ h/\chi + \theta] }$.
In contrast to Euclidean geometry, the rays of different angles started from the same point
may intersect each other.  There is a critical angle $\theta_c$ given by
\[\theta_c = \frac{\pi \kappa}{4-\kappa} = \frac{4 \pi}{\kappa' - 4}\,,\]
such that the flow lines of $e^{i h/\chi}$ and $e^{ i [ h/\chi + \theta] }$ started from a common point
a.s.\ intersect when $\theta<\theta_c$, and a.s.\ do not intersect when $\theta \geq \theta_c$.
If we condition on a flow line~$\eta$, there is a winding height gap in the GFF, in the sense
that $\E[h/\chi\mid\eta]$ just to the right of $\eta$ is larger by $\theta_c$ than the value just to the left of $\eta$.

The east-going flow lines of $e^{i h/\chi}$ started from different points can intersect,
at which point they merge.  The collection of east-going rays started from all points together form a
continuum spanning tree, whose branches are $\SLE_\kappa$'s.  There is a space-filling curve which is
the analog of the UST peano curve, which traces the boundary of this spanning tree, and is a space-filling
version of $\SLE_{\kappa'}$ with $\kappa'=16/\kappa$.

When $\kappa = 4 n /(n+1)$ the critical angle is $\theta_c = n \pi$.
The well-known ``special values'' of $\kappa$ have in the past corresponded to integer values of $n$,
together with the limiting case $n\to\infty$.
For example, $n \in \{1,2,3,5,\infty\}$ gives $\kappa \in \{2, 8/3, 3, 10/3, 4 \}$ and $\kappa' \in \{8, 6, 16/3, 24/5, 4 \}$.
In the case $\kappa' = 12$ and $\kappa = 4/3$, we have $n = 1/2$.

\subsection{Bipolar winding height gap should be \texorpdfstring{$\pi/2$}{\textpi/2}}

To make sense of winding angle in the context of a planar map, one may
view the map as a Riemannian surface obtained by interpreting the
faces as regular unit polygons glued together, and then conformally
map that surface, as in Figure~\ref{planarmapfig}.
\begin{figure}[b!]
\begin{center}
\includegraphics [width=3.5in]{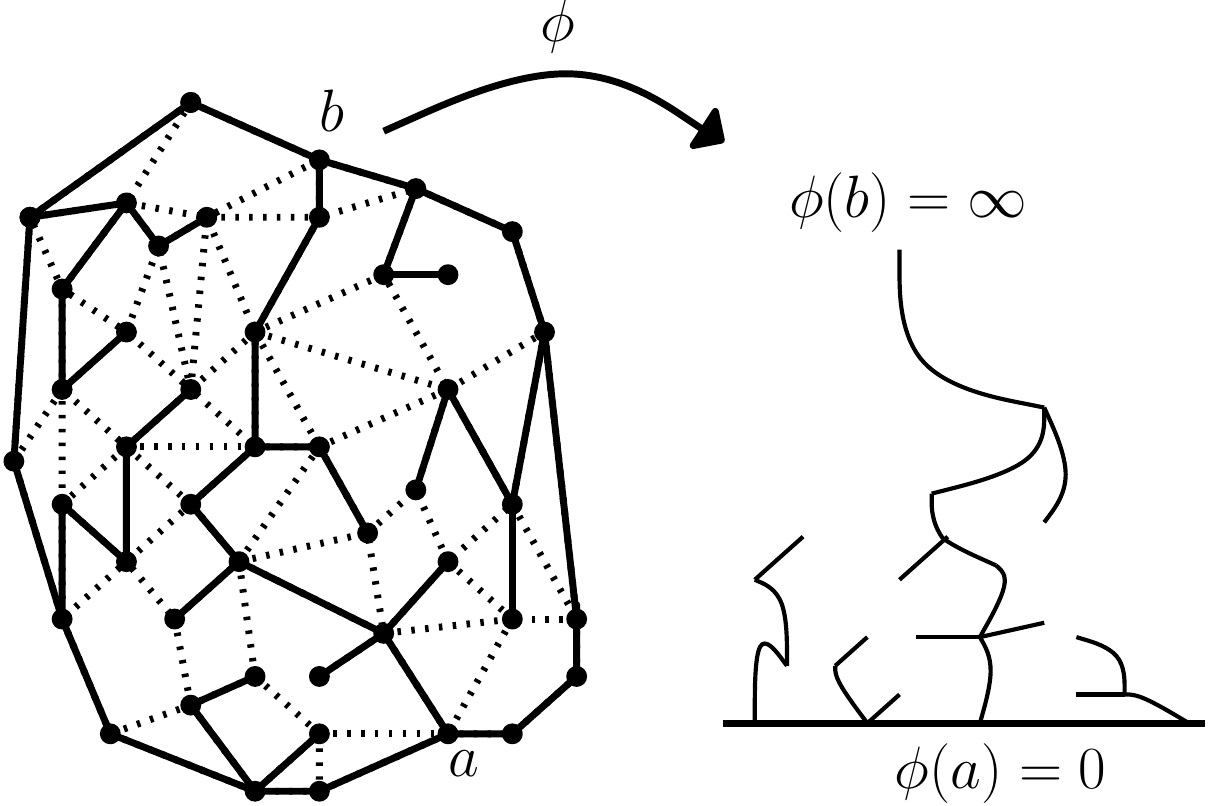}
\end{center}
\caption {\label{planarmapfig} Planar map with a distinguished outer-boundary-plus-one-chord-rooted spanning tree (solid black edges), with chord joining marked boundary points $a$ and $b$, plus image of tree under conformally uniformizing map $\phi$ to $\mathbb H$ (sketch).  (This figure first appeared in \cite{sheffield:welding}.)}
\end{figure}
If a random planar map is decorated by a bipolar orientation, we can assign a winding to every edge that indicates the total amount of right turning (minus left turning) that takes place as one moves along \textit{any\/} north-going path (it doesn't matter which one) from the midpoint of that edge to the north pole.

Consider a NW path started from a vertex incident to the eastern face continued to a vertex incident to the western face.
The portion of the bipolar-oriented map that is south of this path may be east-west reflected to obtain a new bipolar-oriented
planar map.  If this NW path is suitably chosen, so that it is determined by the portion of the map north of it,
then reflecting the portion of the map south of it is a bijection.

This east-west symmetry suggests that the edges to the left of a long NW path have a winding which is, on average,
$\pi/2$ less than the winding of the NW path (as in Figure~\ref{heightgap-bipolar}).

The bipolar-oriented map to the right of the NW path does not have this same reflection symmetry.  Indeed, to the right of the NW path (but not the left), there can be other north-going paths that split off the NW path and rejoin the NW path at another vertex.  Reflecting the bipolar map on the right side of the NW path would then create a cycle.

However, if we reflect the map to the right of the NW path and then reverse the orientations of the edges, no cycles are created, and no new sources or sinks are created except at the endpoints of the NW path.  Thus for a long NW path, we expect the bipolar-oriented map to the right of the path to be approximately reflection-reversal symmetric.

Edges to the right of the NW path may be oriented either toward or away from the path; one expects those oriented away from the path to have smaller winding on average and those oriented toward the path to have larger winding on average. However, by the approximate reflection-reversal symmetry, these two effects should cancel, so that overall there is no expected angle gap between the black path and the red edges to its right (as in Figure~\ref{heightgap-bipolar}).
\enlargethispage{24pt}
\begin {figure}[htbp!]
\begin {center}
\includegraphics[scale=.9]{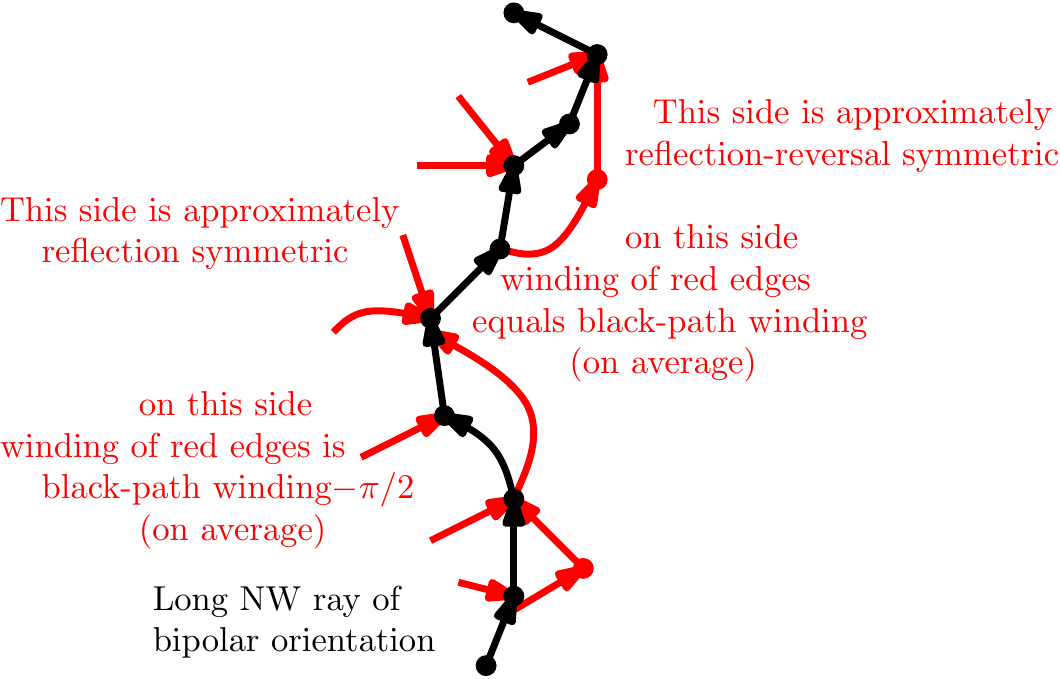}
\caption {\label{heightgap-bipolar}The average height gap on either side of a long NW ray (black) of a bipolar orientation.
On the western side the average height gap is $-\pi/2$, on the eastern side it is $0$.
}
\end {center}
\end {figure}

\section{Open question}

In addition to questions regarding strengthening the topology of convergence, which are discussed at the end of Section~\ref{peanosphere}, it would be interesting to extend the theory to other surface graphs, such as the torus, or a disk with four boundary vertices which are alternately source, sink, source, sink.

\bibliographystyle{hmralphaabbrv}
\addcontentsline{toc}{section}{References}
\small
\bibliography{../bipolar}

\newcommand{\etalchar}[1]{$^{#1}$}
\begin{thebibliography}{CDCH{\etalchar{+}}14}

\bibitem[AHS15]{AHS:twoperspectives}
J.~Aru, Y.~Huang, and X.~Sun.
\newblock Two perspectives of the {2D} unit area quantum sphere and their
  equivalence.
\newblock 2015.
\newblock \arXiv{1512.06190}.

\bibitem[AK15]{abrams-kenyon:fixed-energy}
A.~Abrams and R.~Kenyon.
\newblock Fixed-energy harmonic functions.
\newblock 2015.
\newblock \arXiv{1505.05785}.

\bibitem[Ald91a]{aldous:crt1}
D.~Aldous.
\newblock The continuum random tree. {I}.
\newblock \href{http://dx.doi.org/10.1214/aop/1176990534}{{\em Ann.\ Probab.}
  19(1):1--28}, 1991. \MR{1085326 (91i:60024)}

\bibitem[Ald91b]{aldous:crt2}
D.~Aldous.
\newblock The continuum random tree. {II}. {A}n overview.
\newblock In \href{http://dx.doi.org/10.1017/CBO9780511662980.003}{{\em
  Stochastic analysis}, London Math.\ Soc.\ Lecture Note Ser.\ \#167, pages
  23--70}. Cambridge Univ.\ Press, 1991. \MR{1166406 (93f:60010)}

\bibitem[Ald93]{aldous:crt3}
D.~Aldous.
\newblock The continuum random tree. {III}.
\newblock \href{http://dx.doi.org/10.1214/aop/1176989404}{{\em Ann.\ Probab.}
  21(1):248--289}, 1993. \MR{1207226 (94c:60015)}

\bibitem[BBMF11]{BBMF:baxter-bipolar}
N.~Bonichon, M.~Bousquet-M{\'e}lou, and {\'E}.~Fusy.
\newblock Baxter permutations and plane bipolar orientations.
\newblock {\em S\'em.\ Lothar.\ Combin.} 61A:Article~B61Ah, 2009/11.
  \MR{2734180 (2011m:05023)}

\bibitem[BM11]{bousquet-melou:maps}
M.~Bousquet-M{\'e}lou.
\newblock Counting planar maps, coloured or uncoloured.
\newblock In {\em Surveys in combinatorics 2011}, London Math.\ Soc.\ Lecture
  Note Ser.\ \#392, pages 1--49. Cambridge Univ.\ Press, 2011. \MR{2866730}

\bibitem[Car05]{Cardy}
J.~Cardy.
\newblock S{LE} for theoretical physicists.
\newblock \href{http://dx.doi.org/10.1016/j.aop.2005.04.001}{{\em Ann.\
  Physics} 318(1):81--118}, 2005. \MR{2148644}

\bibitem[CDCH{\etalchar{+}}14]{CDCHKS:ising}
D.~Chelkak, H.~Duminil-Copin, C.~Hongler, A.~Kemppainen, and S.~Smirnov.
\newblock Convergence of {I}sing interfaces to {S}chramm's {SLE} curves.
\newblock \href{http://dx.doi.org/10.1016/j.crma.2013.12.002}{{\em C. R. Math.\
  Acad.\ Sci.\ Paris} 352(2):157--161}, 2014. \MR{3151886}

\bibitem[CV81]{cori-vauquelin}
R.~Cori and B.~Vauquelin.
\newblock Planar maps are well labeled trees.
\newblock \href{http://dx.doi.org/10.4153/CJM-1981-078-2}{{\em Canad.\ J.
  Math.} 33(5):1023--1042}, 1981. \MR{638363 (83c:05070)}

\bibitem[dFdMR95]{bipolarorientationsrevisited}
H.~de~Fraysseix, P.~O. de~Mendez, and P.~Rosenstiehl.
\newblock Bipolar orientations revisited.
\newblock \href{http://dx.doi.org/10.1016/0166-218X(94)00085-R}{{\em Discrete
  Appl.\ Math.} 56(2-3):157--179}, 1995. \MR{1318743 (96i:05073)}

\bibitem[DKRV16]{DKRV:sphere}
F.~David, A.~Kupiainen, R.~Rhodes, and V.~Vargas.
\newblock Liouville quantum gravity on the {R}iemann sphere.
\newblock \href{http://dx.doi.org/10.1007/s00220-016-2572-4}{{\em Comm.\ Math.\
  Phys.} 342(3):869--907}, 2016.
\newblock \arXiv{1410.7318}. \MR{3465434}

\bibitem[DMS14]{DMS:mating}
B.~{Duplantier}, J.~{Miller}, and S.~{Sheffield}.
\newblock {Liouville quantum gravity as a mating of trees}.
\newblock 2014.
\newblock \arXiv{1409.7055}.

\bibitem[DS11]{duplantier-sheffield:KPZ}
B.~Duplantier and S.~Sheffield.
\newblock Liouville quantum gravity and {KPZ}.
\newblock \href{http://dx.doi.org/10.1007/s00222-010-0308-1}{{\em Invent.\
  Math.} 185(2):333--393}, 2011. \MR{2819163 (2012f:81251)}

\bibitem[Dub09]{dubedat:duality}
J.~Dub{\'e}dat.
\newblock Duality of {S}chramm--{L}oewner evolutions.
\newblock {\em Ann.\ Sci.\ \'Ec.\ Norm.\ Sup\'er.~(4)} 42(5):697--724, 2009.
  \MR{2571956 (2011g:60151)}

\bibitem[Dup98]{duplantier:brownian}
B.~Duplantier.
\newblock Random walks and quantum gravity in two dimensions.
\newblock \href{http://dx.doi.org/10.1103/PhysRevLett.81.5489}{{\em Phys.\
  Rev.\ Lett.} 81(25):5489--5492}, 1998. \MR{1666816 (99j:83034)}

\bibitem[DW15]{DurajWachtel}
J.~{Duraj} and V.~{Wachtel}.
\newblock {Invariance principles for random walks in cones}.
\newblock 2015.
\newblock \arXiv{1508.07966}.

\bibitem[FFNO11]{FFNO:baxter}
S.~Felsner, {\'E}.~Fusy, M.~Noy, and D.~Orden.
\newblock Bijections for {B}axter families and related objects.
\newblock \href{http://dx.doi.org/10.1016/j.jcta.2010.03.017}{{\em J. Combin.\
  Theory Ser.~A} 118(3):993--1020}, 2011. \MR{2763051 (2012f:05016)}

\bibitem[FPS09]{FPS}
{\'E}.~Fusy, D.~Poulalhon, and G.~Schaeffer.
\newblock Bijective counting of plane bipolar orientations and {S}chnyder
  woods.
\newblock \href{http://dx.doi.org/10.1016/j.ejc.2009.03.001}{{\em European J.
  Combin.} 30(7):1646--1658}, 2009. \MR{2548656 (2010j:05047)}

\bibitem[Gar09]{Garbit}
R.~Garbit.
\newblock Brownian motion conditioned to stay in a cone.
\newblock {\em J. Math.\ Kyoto Univ.} 49(3):573--592, 2009. \MR{2583602}

\bibitem[GHMS15]{GHMS:covariance}
E.~Gwynne, N.~Holden, J.~Miller, and X.~Sun.
\newblock Brownian motion correlation in the peanosphere for $\kappa > 8$.
\newblock 2015.
\newblock To appear in \textit{Ann.\ Inst.\ Henri Poincar\'e Probab.\ Stat}.
  \arXiv{1510.04687}.

\bibitem[GHS16]{GNS:bipolar}
E.~Gwynne, N.~Holden, and X.~Sun.
\newblock Joint scaling limit of a bipolar-oriented triangulation and its dual
  in the peanosphere sense.
\newblock 2016.
\newblock \arXiv{1603.01194}.

\bibitem[GKMW16]{GKMW:active-tree-map}
E.~Gwynne, A.~Kassel, J.~Miller, and D.~B. Wilson.
\newblock Active spanning trees with bending energy on planar maps and
  {SLE}-decorated {L}iouville quantum gravity for $\kappa>8$.
\newblock 2016.
\newblock \arXiv{1603.09722}.

\bibitem[GM16a]{gwynne-miller:saw}
E.~{Gwynne} and J.~{Miller}.
\newblock {Convergence of the self-avoiding walk on random quadrangulations to
  SLE$\_{8/3}$ on $\sqrt{8/3}$-Liouville quantum gravity}.
\newblock {\em ArXiv e-prints} , August 2016, \arxiv{1608.00956}.

\bibitem[GM16b]{gwynne-miller:topology}
E.~Gwynne and J.~Miller.
\newblock Convergence of the topology of critical {F}ortuin--{K}asteleyn planar
  maps to that of {CLE}$_\kappa$ on a {L}iouville quantum surface.
\newblock 2016.
\newblock In preparation.

\bibitem[GM16c]{gwynne-miller:gluing}
E.~{Gwynne} and J.~{Miller}.
\newblock {Metric gluing of Brownian and $\sqrt{8/3}$-Liouville quantum gravity
  surfaces}.
\newblock {\em ArXiv e-prints} , August 2016, \arxiv{1608.00955}.

\bibitem[GM16d]{gwynne-miller:uihpq}
E.~{Gwynne} and J.~{Miller}.
\newblock {Scaling limit of the uniform infinite half-plane quadrangulation in
  the Gromov-Hausdorff-Prokhorov-uniform topology}.
\newblock {\em ArXiv e-prints} , August 2016, \arxiv{1608.00954}.

\bibitem[GMS15]{GMS:cone_times}
E.~{Gwynne}, C.~{Mao}, and X.~{Sun}.
\newblock {Scaling limits for the critical Fortuin--Kasteleyn model on a random
  planar map {I}: cone times}.
\newblock 2015.
\newblock \arXiv{1502.00546}.

\bibitem[GS15a]{gwynne-sun:finite_volume}
E.~{Gwynne} and X.~{Sun}.
\newblock {Scaling limits for the critical {F}ortuin--{K}astelyn model on a
  random planar map {II}: local estimates and empty reduced word exponent}.
\newblock 2015.
\newblock \arXiv{1505.03375}.

\bibitem[GS15b]{gwynne-sun:finite}
E.~Gwynne and X.~Sun.
\newblock Scaling limits for the critical {F}ortuin--{K}astelyn model on a
  random planar map {III}: finite volume case.
\newblock 2015.
\newblock \arXiv{1510.06346}.

\bibitem[Ken00]{dominoconformal}
R.~Kenyon.
\newblock Conformal invariance of domino tiling.
\newblock \href{http://dx.doi.org/10.1214/aop/1019160260}{{\em Ann.\ Probab.}
  28(2):759--795}, 2000. \MR{1782431 (2002e:52022)}

\bibitem[KMSW16]{KMSW2}
R.~W. Kenyon, J.~Miller, S.~Sheffield, and D.~B. Wilson.
\newblock The six-vertex model and {S}chramm--{L}oewner evolution.
\newblock 2016.
\newblock \arXiv{1605.06471}.

\bibitem[KPW00]{KPW}
R.~W. Kenyon, J.~G. Propp, and D.~B. Wilson.
\newblock Trees and matchings.
\newblock {\em Electron.\ J. Combin.} 7:Research Paper 25, 34 pp., 2000.
  \MR{1756162 (2001a:05123)}

\bibitem[KW16]{kassel-wilson:active}
A.~Kassel and D.~B. Wilson.
\newblock Active spanning trees and {S}chramm--{L}oewner evolution.
\newblock \href{http://dx.doi.org/10.1103/PhysRevE.93.062121}{{\em Phys.\ Rev.\
  E} 93:062121}, 2016.

\bibitem[LEC67]{lempel-even-cederbaum}
A.~Lempel, S.~Even, and I.~Cederbaum.
\newblock An algorithm for planarity testing of graphs.
\newblock In {\em Theory of {G}raphs ({I}nternat.\ {S}ympos., {R}ome, 1966)},
  pages 215--232. 1967. \MR{0220617}

\bibitem[LG13]{legall:bm}
J.-F. Le~Gall.
\newblock Uniqueness and universality of the {B}rownian map.
\newblock \href{http://dx.doi.org/10.1214/12-AOP792}{{\em Ann.\ Probab.}
  41(4):2880--2960}, 2013. \MR{3112934}

\bibitem[LSW01a]{LSW:brownian-frontier}
G.~F. Lawler, O.~Schramm, and W.~Werner.
\newblock The dimension of the planar {B}rownian frontier is {$4/3$}.
\newblock \href{http://dx.doi.org/10.4310/MRL.2001.v8.n4.a1}{{\em Math.\ Res.\
  Lett.} 8(4):401--411}, 2001. \MR{1849257 (2003a:60127b)}

\bibitem[LSW01b]{LSW:brownian1}
G.~F. Lawler, O.~Schramm, and W.~Werner.
\newblock Values of {B}rownian intersection exponents. {I}. {H}alf-plane
  exponents.
\newblock \href{http://dx.doi.org/10.1007/BF02392618}{{\em Acta Math.}
  187(2):237--273}, 2001. \MR{1879850 (2002m:60159a)}

\bibitem[LSW01c]{LSW:brownian2}
G.~F. Lawler, O.~Schramm, and W.~Werner.
\newblock Values of {B}rownian intersection exponents. {II}. {P}lane exponents.
\newblock \href{http://dx.doi.org/10.1007/BF02392619}{{\em Acta Math.}
  187(2):275--308}, 2001. \MR{1879851 (2002m:60159b)}

\bibitem[LSW02]{LSW:brownian3}
G.~F. Lawler, O.~Schramm, and W.~Werner.
\newblock Values of {B}rownian intersection exponents. {III}. {T}wo-sided
  exponents.
\newblock \href{http://dx.doi.org/10.1016/S0246-0203(01)01089-5}{{\em Ann.\
  Inst.\ H. Poincar\'e Probab.\ Statist.} 38(1):109--123}, 2002. \MR{1899232
  (2003d:60163)}

\bibitem[LSW04]{LSW:tree}
G.~F. Lawler, O.~Schramm, and W.~Werner.
\newblock Conformal invariance of planar loop-erased random walks and uniform
  spanning trees.
\newblock \href{http://dx.doi.org/10.1214/aop/1079021469}{{\em Ann.\ Probab.}
  32(1B):939--995}, 2004. \MR{2044671 (2005f:82043)}

\bibitem[Mie13]{miermont:bm}
G.~Miermont.
\newblock The {B}rownian map is the scaling limit of uniform random plane
  quadrangulations.
\newblock \href{http://dx.doi.org/10.1007/s11511-013-0096-8}{{\em Acta Math.}
  210(2):319--401}, 2013. \MR{3070569}

\bibitem[Moo25]{moore:spheres}
R.~L. Moore.
\newblock Concerning upper semi-continuous collections of continua.
\newblock \href{http://dx.doi.org/10.2307/1989234}{{\em Trans.\ Amer.\ Math.\
  Soc.} 27(4):416--428}, 1925. \MR{1501320}

\bibitem[MS12a]{miller-sheffield:ig1}
J.~Miller and S.~Sheffield.
\newblock Imaginary geometry {I}: Interacting {SLE}s.
\newblock 2012.
\newblock To appear in \textit{Probab.\ Theory Related Fields}.
  \arXiv{1201.1496}.

\bibitem[MS12b]{miller-sheffield:ig3}
J.~Miller and S.~Sheffield.
\newblock Imaginary geometry {III}: reversibility of {SLE}$_\kappa$ for $\kappa
  \in (4,8)$.
\newblock 2012.
\newblock To appear in \textit{Ann.\ of Math.~(2)}. \arXiv{1201.1498}.

\bibitem[MS13a]{miller-sheffield:ig4}
J.~Miller and S.~Sheffield.
\newblock Imaginary geometry {IV}: interior rays, whole-plane reversibility,
  and space-filling trees.
\newblock 2013.
\newblock \arXiv{1302.4738}.

\bibitem[MS13b]{miller-sheffield:qle}
J.~Miller and S.~Sheffield.
\newblock Quantum {L}oewner evolution.
\newblock 2013.
\newblock To appear in \textit{Duke Math.\ J}. \arXiv{1312.5745}.

\bibitem[MS15a]{miller-sheffield:axiomatic}
J.~Miller and S.~Sheffield.
\newblock An axiomatic characterization of the {B}rownian map.
\newblock 2015.
\newblock \arXiv{1506.03806}.

\bibitem[MS15b]{miller-sheffield:metric}
J.~{Miller} and S.~{Sheffield}.
\newblock Liouville quantum gravity and the {B}rownian map {I}: The
  {QLE(8/3,0)} metric.
\newblock 2015.
\newblock \arXiv{1507.00719}.

\bibitem[MS15c]{miller-sheffield:finite-trees}
J.~Miller and S.~Sheffield.
\newblock Liouville quantum gravity spheres as matings of finite-diameter
  trees.
\newblock 2015.
\newblock \arXiv{1506.03804}.

\bibitem[MS16a]{miller-sheffield:ig2}
J.~Miller and S.~Sheffield.
\newblock Imaginary geometry {II}: reversibility of
  {SLE}$_\kappa(\rho_1;\rho_2)$ for $\kappa \in (0,4)$.
\newblock \href{http://dx.doi.org/10.1214/14-AOP943}{{\em Ann.\ Probab}
  44(3):1647--1722}, 2016.

\bibitem[MS16b]{miller-sheffield:continuity}
J.~Miller and S.~Sheffield.
\newblock Liouville quantum gravity and the {B}rownian map {II}: geodesics and
  continuity of the embedding.
\newblock 2016.
\newblock \arXiv{1605.03563}.

\bibitem[MS16c]{miller-sheffield:determined}
J.~Miller and S.~Sheffield.
\newblock Liouville quantum gravity and the {B}rownian map {III}: the conformal
  structure is determined.
\newblock 2016.
\newblock In preparation.

\bibitem[Mul67]{mullin:tree-maps}
R.~C. Mullin.
\newblock On the enumeration of tree-rooted maps.
\newblock \href{http://dx.doi.org/10.4153/cjm-1967-010-x}{{\em Canad.\ J.
  Math.} 19:174--183}, 1967. \MR{0205882 (34 \#5708)}

\bibitem[RS02]{rs:k_8_reverse}
S.~Rohde and O.~Schramm.
\newblock 2002.
\newblock Private communication.

\bibitem[Sch98]{schaeffer}
G.~Schaeffer.
\newblock {\em
  \href{http://www.lix.polytechnique.fr/~schaeffe/Biblio/PhD-Schaeffer.pdf}{Conjugaison
  d'arbres et cartes combinatoires al\'eatoires}}.
\newblock PhD thesis, Universit\'e Bordeaux I, 1998.

\bibitem[Sch00]{schramm:sle}
O.~Schramm.
\newblock Scaling limits of loop-erased random walks and uniform spanning
  trees.
\newblock \href{http://dx.doi.org/10.1007/BF02803524}{{\em Israel J. Math.}
  118:221--288}, 2000. \MR{1776084 (2001m:60227)}

\bibitem[She09]{sheffield:cle-trees}
S.~Sheffield.
\newblock Exploration trees and conformal loop ensembles.
\newblock \href{http://dx.doi.org/10.1215/00127094-2009-007}{{\em Duke Math.\
  J.} 147(1):79--129}, 2009. \MR{2494457 (2010g:60184)}

\bibitem[She10]{sheffield:welding}
S.~Sheffield.
\newblock Conformal weldings of random surfaces: {SLE} and the quantum gravity
  zipper.
\newblock 2010.
\newblock To appear in \textit{Ann.\ Probab}. \arXiv{1012.4797}.

\bibitem[She11]{sheffield:burgers}
S.~Sheffield.
\newblock Quantum gravity and inventory accumulation.
\newblock 2011.
\newblock To appear in \textit{Ann.\ Probab}. \arXiv{1108.2241}.

\bibitem[Shi85]{shimura:cone}
M.~Shimura.
\newblock Excursions in a cone for two-dimensional {B}rownian motion.
\newblock {\em J. Math. Kyoto Univ.} 25(3):433--443, 1985. \MR{807490}

\bibitem[Smi01]{smirnov:percolation}
S.~Smirnov.
\newblock Critical percolation in the plane: conformal invariance, {C}ardy's
  formula, scaling limits.
\newblock \href{http://dx.doi.org/10.1016/S0764-4442(01)01991-7}{{\em C. R.
  Acad.\ Sci.\ Paris S\'er.~I Math.} 333(3):239--244}, 2001. \MR{1851632
  (2002f:60193)}

\bibitem[Smi10]{smirnov:ising}
S.~Smirnov.
\newblock Conformal invariance in random cluster models. {I}. {H}olomorphic
  fermions in the {I}sing model.
\newblock \href{http://dx.doi.org/10.4007/annals.2010.172.1441}{{\em Ann.\ of
  Math.\ (2)} 172(2):1435--1467}, 2010. \MR{2680496 (2011m:60302)}

\bibitem[SS09]{schramm-sheffield:discrete-GFF}
O.~Schramm and S.~Sheffield.
\newblock Contour lines of the two-dimensional discrete {G}aussian free field.
\newblock \href{http://dx.doi.org/10.1007/s11511-009-0034-y}{{\em Acta Math.}
  202(1):21--137}, 2009. \MR{2486487 (2010f:60238)}

\bibitem[SS13]{schramm-sheffield:continuum-GFF}
O.~Schramm and S.~Sheffield.
\newblock A contour line of the continuum {G}aussian free field.
\newblock \href{http://dx.doi.org/10.1007/s00440-012-0449-9}{{\em Probab.\
  Theory Related Fields} 157(1-2):47--80}, 2013. \MR{3101840}

\bibitem[SW12]{sheffield-werner:markovian}
S.~Sheffield and W.~Werner.
\newblock Conformal loop ensembles: the {M}arkovian characterization and the
  loop-soup construction.
\newblock \href{http://dx.doi.org/10.4007/annals.2012.176.3.8}{{\em Ann.\ of
  Math.\ (2)} 176(3):1827--1917}, 2012. \MR{2979861}

\bibitem[Tut63]{tutte:planar-maps}
W.~T. Tutte.
\newblock A census of planar maps.
\newblock \href{http://dx.doi.org/10.4153/cjm-1963-029-x}{{\em Canad.\ J.
  Math.} 15:249--271}, 1963. \MR{0146823 (26 \#4343)}

\bibitem[Tut73]{tutte:triangulations}
W.~T. Tutte.
\newblock Chromatic sums for rooted planar triangulations: the cases {$\lambda
  =1$} and {$\lambda =2$}.
\newblock \href{http://dx.doi.org/10.4153/cjm-1973-043-3}{{\em Canad.\ J.
  Math.} 25:426--447}, 1973. \MR{0314677}

\bibitem[Zha08]{zhan:duality}
D.~Zhan.
\newblock Duality of chordal {SLE}.
\newblock \href{http://dx.doi.org/10.1007/s00222-008-0132-z}{{\em Invent.\
  Math.} 174(2):309--353}, 2008. \MR{2439609 (2010f:60239)}

\end{thebibliography}

\end{document}